\documentclass[11pt,a4paper]{article}

\usepackage[english]{babel}
\usepackage{latexsym}

\usepackage{amsmath,amsfonts,amsthm,amssymb,amscd}

\usepackage[mathscr]{eucal}

\usepackage{enumerate}

\usepackage{graphicx}
\usepackage{fullpage}
\usepackage{url}
\usepackage{color}

\theoremstyle{definition}
\newcounter{Polya}

\newtheorem{theo}{Theorem}[section]
\newtheorem{pro}[theo]{Proposition}
\newtheorem{lemma}[theo]{Lemma}
\newtheorem{cor}[theo]{Corollary}
\newtheorem{defi}[theo]{Definition}
\newtheorem{obs}[theo]{Observation}
\newtheorem{rema}[theo]{Remark}
\newtheorem{prp}[theo]{Properties}
\newtheorem{exam}[theo]{Example}
\newtheorem{prps}[theo]{Propiedades}
\newtheorem{obss}[theo]{Observations}
\newtheorem{ejes}[theo]{Ejemplos}

\swapnumbers

\newenvironment{prop}{\smallskip\begin{pro}}{\end{pro}\smallskip}

\newenvironment{coro}{\smallskip\begin{cor}}{\end{cor}\smallskip}
\newenvironment{den}{\smallskip\begin{defi}}{\end{defi}\smallskip}

\def\fy{\varphi}
\def\oo{\infty}
\def\ri{\rightarrow}
\def\a{\alpha}
\def\b{\beta}
\def\de{\delta}

\def\ep{\varepsilon}
\def\o{\omega}

\def\ga{\gamma}

\def\ro{\rho}
\def\th{\theta}
\def\Bet{\mathfrak{B}}
\def\bl{\boldsymbol{\ell}}
\def\en{\subseteq}
\def\N{\mathbb{N}}
\def\R{\mathbb{R}}

\def\Ri{{\cal R}}
\def\M{\mathbb{M}}

\def\bM{\mathbb{M}}
\def\L{\mathbb{L}} 

\def\m{\boldsymbol{m}}

\def\bm{\boldsymbol{m}}

\def\no{\nonumber}

\begin{document}

\title{
Log-convex sequences and nonzero proximate orders}

\author{Javier Jim\'enez-Garrido, Javier Sanz and Gerhard Schindl}
\date{\today}

\maketitle

\abstract{Summability methods for ultraholomorphic classes in sectors, defined in terms of a strongly regular
sequence $\M=(M_p)_{p\in\N_0}$, have been put forward by A. Lastra, S. Malek and the second author~\cite{lastramaleksanz3}, and their validity depends on the possibility of associating to $\M$ a nonzero proximate order. We provide several characterizations of this and other related properties, in which the concept of regular variation for functions and sequences plays a prominent role. In particular, we show how to construct well-behaved strongly regular sequences from nonzero proximate orders.
}

\par

\section{Introduction}

A general, common treatment of summability in Carleman ultraholomorphic classes in sectors, which extends the powerful theory of $k-$summability dealing with Gevrey formal power series, has been put forward by A. Lastra, S. Malek and the second author~\cite{lastramaleksanz3}. The technique consisted in the construction of pairs of kernel functions with suitable asymptotic and growth properties, in terms of which to define formal and analytic Laplace- and Borel-like transforms which allow one to explicitly recover the sum of a summable formal power series in a direction. The main inspiration came from the theory of moment summability methods developed by W.~Balser in~\cite[Section\ 5.5]{Balserutx}, and required the notion of proximate order, appearing in the study of growth properties of holomorphic functions in sectors.

The Carleman ultraholomorphic classes $\tilde{\mathcal{A}}_{\M}(G)$ we deal with are those consisting of holomorphic functions $f$ admitting an asymptotic expansion $\hat f=\sum_{p\ge 0}a_pz^p$, in a sectorial region $G$ with vertex at 0, with remainders suitably bounded in terms of a sequence $\M=(M_p)_{p\in\N_0}$ of positive real numbers: for every bounded and proper subsector $T$ of $G$, there
exist $C_T,A_T>0$ such that for every $p\in\N_0$ and $z\in T$, one has
\begin{equation*}\Big|f(z)-\sum_{k=0}^{p-1}a_kz^k \Big|\le C_TA_T^pM_{p}|z|^p.
\end{equation*}

We will mostly consider logarithmically convex sequences $\M$ with quotients $m_p:=M_{p+1}/M_p$ tending to infinity, and frequently our attention will focus on
strongly regular sequences as defined by V. Thilliez~\cite{thilliez},
of which the best known example is that of Gevrey classes, appearing when the sequence is chosen to be $(p!^{\a})_{p\in\N_{0}}$, $\a>0$.

The second author introduced in~\cite{SanzFlat} the constant \begin{equation*}
\omega(\M)=\liminf_{p\to\infty}
\frac{\log(m_{p})}{\log(p)}\in(0,\infty),
\end{equation*}
measuring the rate of growth of any strongly regular sequence $\M$. Whenever the associated function $d_{\M}(t)=\log(M(t))/\log t$, where
\begin{equation}\label{eqDefinMdet}
M(t):=\sup_{p\in\N_{0}}\log\big(\frac{t^p}{M_{p}}\big)
,\quad t>0,
\end{equation}
is a nonzero proximate order, one can provide nontrivial flat functions in optimal sectors of opening $\pi\,\omega(\M)$, and this is the crucial point for the success in putting forward a satisfactory summability theory in $\tilde{\mathcal{A}}_{\M}(G)$ (see~\cite{lastramaleksanz3}). So, it seemed important to characterize the fact that $d_{\M}$ is a nonzero proximate order in a simple way, what was achieved in the paper~\cite{JimenezSanz} by the first two authors. The main result was the following.

\begin{theo}[\cite{JimenezSanz},\ Th.\ 3.14]\label{theorem.condition3.implies.beta.converge}
    Let $\M$ be a strongly regular sequence, then the following are equivalent:
\begin{enumerate}[(i)]
 \item $d_{\M}(t)$ is a proximate order,
 \item $\lim_{t\ri\oo}d_{\M}(t)=1/\o(\M)$,
 \item $\lim_{p\ri\oo} \log(m_p)/\log(p)=\o(\M)$,
 \item  $\lim_{p\ri\oo} \log\big(m_p/M_p^{1/p}\big)=\o(\M)$.
\end{enumerate}
\end{theo}

These conditions are satisfied for every strongly regular sequence appearing in the applications (ODEs, PDEs and difference equations; see the introduction of \cite{JimenezSanz} and the references therein for a non-complete account). However, it was not clear for us whether these conditions held true for any strongly regular sequence, so we investigated on some pathological examples. To our surprise, we found an example (see Example~\ref{example.dM.not.p.o} in this paper) satisfying (iii) but not (iv). It turns out that in our (wrong) proof that (iii) implies (iv), a very recent criterion by F. Moricz~\cite{Moricz} for the convergence of a sequence summable by Riesz means played a prominent role.

\begin{theo}[\cite{Moricz},\ Th.\ 5.1]\label{theoMoricz}
If a sequence $(s_k)_{k\in\N}$ of real numbers is Riesz-summable to some $A\in\R$, then the ordinary limit exists (with the same value) if, and only if,
\begin{equation}\label{eqmoricz1}
 \limsup_{\lambda\ri1^+}\liminf_{p\ri\oo} \frac{1}{(\lfloor p^\lambda\rfloor-p)H_p}\sum_{k=p+1}^{\lfloor p^\lambda\rfloor}\frac{s_k-s_p}{k}\geq 0
\end{equation}
and
\begin{equation}\label{eqmoricz2}
 \limsup_{\lambda\ri1^-}\liminf_{p\ri\oo} \frac{1}{(p-\lfloor p^\lambda\rfloor )H_p}\sum^{p}_{k= \lfloor p^\lambda\rfloor+1}\frac{s_p-s_k}{k}\geq 0,
\end{equation}
where $\lfloor \cdot \rfloor$ denotes the integer part, and $H_p$ is the $p-$th partial sum of the harmonic series.

\end{theo}

By carefully inspecting this result, we noticed that its statement is not correct, although a right one may be deduced from the final lemma 5.5 in Moricz's paper, where the expressions in \eqref{eqmoricz1} and \eqref{eqmoricz2} are rewritten. Indeed, they should read as follows:
\begin{equation*}
 \limsup_{\lambda\ri1^+}\liminf_{p\ri\oo} \frac{1}{H_{\lfloor p^\lambda\rfloor}-H_p}\sum_{k=p+1}^{\lfloor p^\lambda\rfloor}\frac{s_k-s_p}{k}\geq 0
\end{equation*}
and
\begin{equation*}
 \limsup_{\lambda\ri1^-}\liminf_{p\ri\oo} \frac{1}{H_p-H_{\lfloor p^\lambda\rfloor}}\sum^{p}_{k= \lfloor p^\lambda\rfloor+1}\frac{s_p-s_k}{k}\geq 0.
\end{equation*}

This fact made clear to us that several implications in Theorem~\ref{theorem.condition3.implies.beta.converge} were false. In Section~\ref{sectSequencesDefineProxOrder} the suitable corrections will be carefully described (see, in particular, Remark~\ref{remaCorrections}).

Apart from the necessity to recover from this mistake, there are some important points to note. Firstly, as indicated in~\cite[Remark\ 4.11(iii)]{SanzFlat}, in order to obtain the summability theory in $\tilde{\mathcal{A}}_{\M}(G)$ mentioned above, it is not crucial that $d_{\M}$ is a nonzero proximate order, but rather that\par
(*) there exist a nonzero proximate order $\rho(t)$ and constants $A,B>0$ such that
\begin{equation*}
A\le \log(t)\big(d_{\M}(t)-\rho(t)\big)\le B\qquad \textrm{for $t$ large enough,}
\end{equation*}
since these estimates allow for the obtention of the required kernels, integral transforms and asymptotic relations. Secondly, and in the same line of ideas, there is some flexibility in the definition of the space $\tilde{\mathcal{A}}_{\M}(G)$: If $\L=(L_p)_{p\in\N_0}$ is another sequence of positive real numbers and it is equivalent to $\M$ (in the sense that there exist $C,D>0$ such that $D^p L_p \leq M_p \leq C^p L_p$ for every $p \in \N_0$), then $\tilde{\mathcal{A}}_{\M}(G)=\tilde{\mathcal{A}}_{\L}(G)$. So, even if $d_{\M}$ is not a nonzero proximate order, it makes sense to wonder whether\par
(**) there exists a sequence $\L$ equivalent to $\M$ and such that $d_{\L}$ is a nonzero proximate order.

The main aim of this paper is to provide statements, as accurate as possible (in the sense that they impose the least restrictive hypotheses on the sequence $\M$), clarifying the equivalences or implications between the different properties (i), (ii), (iii), (iv), (*) and (**), together with the property of regular variation of the sequence of quotients, which already appeared in~\cite{JimenezSanz}. As a byproduct we obtain a new characterization of regularly varying sequences.

It will be specially interesting the fact that, following a classical idea of H. Komatsu~\cite{komatsu}, one can also go from nonzero proximate orders to well-behaved strongly regular sequences, so reversing the way from suitable sequences $\M$ to proximate orders $d_{\M}$. In this respect the results of L.S. Maergoiz~\cite{Maergoiz}, on the construction of holomorphic functions in sectors whose restriction to the positive real axis has a growth accurately specified by a given nonzero proximate order, will be extremely useful.

We will also show several examples that prove that, in some cases, our results are sharp. In particular, Example~\ref{example.dM.not.p.o} provides a strongly regular sequence not satisfying (iv) which, nevertheless, is equivalent to a Gevrey sequence (for which the corresponding function $d$ is known to be a proximate order). Example~\ref{exampleGdoesnotimplyE} gives a strongly regular sequence $\M$ for which the limit in (iii) exists and with equal indices $\o(\M)$ and $\ga(\M)$ (this last constant was introduced by V. Thilliez~\cite{thilliez}) and which, however, does not admit a proximate order. Finally, Example~\ref{exampleNoAdmiteOrden} shows a strongly regular sequence for which the limit in (iii) does not exist, so solving another open question.

\section{Preliminaries}

This section is devoted to provide all the necessary information regarding proximate orders and logarithmically convex sequences.

\subsection{Proximate orders}

We recall the notion of proximate orders, appearing in the theory of growth of entire functions and developed, among others, by E. Lindel\"of, G. Valiron, B.Ja. Levin, A.A. Goldberg, I.V. Ostrosvkii and L.S. Maergoiz (see the references~\cite{Valiron42,Levin,GoldbergOstrowskii,Maergoiz}).

\begin{den}\label{OAD:1}
We say a real function $\ro(r)$ defined on $(c,\oo)$  is a {\it proximate order},
if the following hold:
 \begin{enumerate}[(A)]
  \item $\ro$ is continuous and piecewise continuously differentiable in $(c,\oo)$ (meaning that it is differentiable except possibly at a sequence of points, tending to infinity, at any of which it is continuous and has distinct finite lateral derivatives),\label{OA1:1}
  \item $\ro(r) \geq 0$ for every $r>0$,\label{OA2:1}
  \item $\lim_{r \ri \oo} \ro(r)=\ro< \oo$, \label{OA3:1}
  \item $\lim_{r  \ri \oo} r \ro'(r) \log(r) = 0$. \label{OA4:1}
 \end{enumerate}
In case the value $\rho$ in (C) is positive (respectively, is 0), we say $\rho(r)$ is a \textit{nonzero} (resp. \textit{zero}) proximate order.
\end{den}

\begin{den}\label{defiEquivalentProxOrders}
Two proximate orders $\ro_1(r)$ and $\ro_2(r)$ are said to be {\it equivalent} if
 \begin{equation}
  \lim_{r \ri \oo} (\ro_1(r) - \ro_2(r)) \log(r) = 0 . \no
 \end{equation}
\end{den}

\begin{rema}\label{remaEquivProxOrder}
The equivalence means precisely that the functions $V_1(r)=r^{\ro_1(r)}$ and $V_2(r)=r^{\ro_2(r)}$ are equivalent in the classical sense, i.e.,
$$\lim_{r\ri\oo} \frac{V_1(r)}{V_2(r)}=
\lim_{r\ri\oo}\frac{r^{\ro_1(r)}}{r^{\ro_2(r)}} = 1.$$
Moreover, it implies that $\lim_{r\ri\oo}\ro_1(r) = \lim_{r\ri\oo}\ro_2(r)$, and so they are simultaneously nonzero or not.
\end{rema}

\begin{exam}\label{exampleOrders}
The following are examples of proximate orders, when defined in suitable intervals $(c,\infty)$:
\begin{itemize}
\item[(i)] $\ro_{\a,\b}(t)=\displaystyle\frac{1}{\a}-\frac{\b}{\a}\frac{\log(\log(t))}{\log(t)}$, $\a>0$, $\b\in\R$.
\item[(ii)] $\rho(t)=\rho+\displaystyle\frac{1}{t^\ga}$, $\ro\ge 0$, $\ga>0$.
\item[(iii)] $\rho(t)=\rho+\displaystyle\frac{1}{\log^\ga(t)}$, $\ro\ge 0$, $\ga>0$.
\end{itemize}
An example of a function verifying all the conditions except (D) is $\ro(t)=\ro+\sin(t)/t$.
\end{exam}

The following statement establishes an important connection, as it will be seen in the forthcoming results.

\begin{prop}[\cite{bingGoldTeug}, Prop.\ 7.4.1]\label{propRegVarOrdAprox}
Let $\ro(r)$ be a proximate order. Then, the function $V(r)=r^{\ro(r)}$ is \textit{regularly varying}, that is,
\begin{equation*}
  \lim_{r \to \infty} \frac{V(tr)}{V(r)}= t^{\ro},
  \end{equation*}
uniformly in the compact sets of $(0,\infty)$.
\end{prop}

\begin{rema}\label{remaInverseRegularlyVarying}
Suppose $\ro(r)$ $(r\geq c\geq0)$ is a proximate order tending to $\ro>0$ at infinity. Then the function $V(r):=r^{\ro(r)}$ is strictly increasing for $r>R$, where $R$ is large enough.
The inverse function $r=U(s)$, $s>V(R)$, has the property that the function $\ro^*(s)=\log( U(s) )/ \log (s)$ is a proximate order
and $\ro^*(s)\ri 1/\ro$ as $s\ri\oo$ (see Property 1.8 in~\cite{Maergoiz}). This $\ro^*(s)$ is called the \textit{proximate order conjugate to $\ro(r)$}. Note that, by Proposition~\ref{propRegVarOrdAprox}, the function $U$ is regularly varying.
\end{rema}

Let $\ga$ be a positive real number. We consider the regions in the Riemann surface of the logarithm $\Ri$ given by
  \begin{equation*}
   L(\ga)=\{ (r,\th) \in \Ri : |\theta|<\ga,\ r>0 \}.
  \end{equation*}

The following result of L.S. Maergoiz will be important later on. For an
arbitrary sector bisected by the positive real axis, it provides holomorphic functions whose
restriction to $(0,\oo)$ is real and has a growth at infinity specified by a prescribed nonzero proximate
order. These functions were used to construct nontrivial flat functions and kernels of summability in Carleman ultraholomorphic classes by the second author (see~\cite{SanzFlat}).

\begin{theo}[\cite{Maergoiz}, Th.\ 2.4]\label{theorpropanalproxorde}
Let $\ro(r)$ be a proximate order with $\ro(r)\to\ro>0$ as $r\to \infty$. For every $\ga>0$ there exists an analytic
function $V(z)$ in $L(\ga)$ such that:
  \begin{enumerate}[(I)]
  \item  For every $z \in L(\ga)$,
 \begin{equation*}
    \lim_{r \to \infty} \frac{V(zr)}{V(r)}= z^{\ro},
  \end{equation*}
uniformly in the compact sets of $L(\ga)$.
\item $\overline{V(z)}=V(\overline{z})$ for every $z \in L(\ga)$, where, for $z=(|z|,\arg(z))$, we
put $\overline{z}=(|z|,-\arg(z))$.
\item $V(r)$ is positive in $(0,\infty)$, strictly increasing and $\lim_{r\to 0}V(r)=0$.
\item The function $t\in\R\to V(e^t)$ is strictly convex (i.e. $V$ is strictly convex relative to $\log(r)$).
\item The function $\log(V(r))$ is strictly concave in $(0,\infty)$.
\item  The function $\ro_V(r):=\log( V(r))/\log(r)$, $r>0$, is a proximate order equivalent to $\ro(r)$.
    \end{enumerate}
\end{theo}

 \begin{defi}\label{defClassMaergoiz}
 Let $\ga>0$ and $\ro(r)$ be a proximate order, $\ro(r) \ri \ro >0$.  $\Bet(\ga,\ro(r))$ stands for the class of functions $V(z)$ defined in $L(\ga)$ satisfying the conditions (I)-(VI) of Theorem~\ref{theorpropanalproxorde}.

 \end{defi}

 \subsection{Logarithmically convex sequences}

In what follows, $\M=(M_p)_{p\geq 0}$ always stands for
a sequence of positive real numbers, and we always assume that $M_0=1$.

\begin{defi}\label{defPropertiesSequences}
We say:
\begin{itemize}
\item[(i)]  $\M$ is \textit{logarithmically convex} (for short, (lc)) if
$$M_{p}^{2}\le M_{p-1}M_{p+1},\qquad p\in\N.
$$

\item[(ii)]  $\M$ is of \textit{moderate growth} (briefly, (mg)) if there exists $A>0$ such that
$$M_{p+q}\le A^{p+q}M_{p}M_{q},\qquad p,q\in\N_0.$$

\item[(iii)]  $\bM$ satisfies the \textit{strong non-quasianalyticity condition} (for short, (snq)) whenever there exists $B>0$ such that
$$
\sum_{q\ge p}\frac{M_{q}}{(q+1)M_{q+1}}\le B\frac{M_{p}}{M_{p+1}},\qquad p\in\N_0.
$$
\end{itemize}
\end{defi}

\begin{defi} For a sequence $\M$ we define {\it the sequence of quotients} $\m=(m_p)_{p\in\N_0}$ by $$m_p:=\frac{M_{p+1}}{M_p} \qquad p\in \N_0.$$
\end{defi}

Although in many of our results we depart from a (lc) sequence with quotients tending to infinity, we will frequently deduce that we are dealing with strongly regular sequences.

\begin{defi}[\cite{thilliez}]
A sequence $\M$ is {\it strongly regular} if it verifies the properties (i), (ii) and (iii) in Definition~\ref{defPropertiesSequences}.
\end{defi}

\begin{rema}\label{remaequivMm}
Observe that for every $p\in\N$ one has
\begin{equation}\label{eqMfromm}
M_p=\frac{M_p}{M_{p-1}}\frac{M_{p-1}}{M_{p-2}}\dots\frac{M_2}{M_{1}}\frac{M_1}{M_{0}}= m_{p-1}m_{p-2}\dots m_1m_0.
\end{equation}
So, one may recover the sequence $\M$ (with $M_0=1$) once $\bm$ is known, and hence the knowledge of one of the sequences amounts to that of the other.
Sequences of quotients of sequences $\M$, $\L$, etc. will be denoted by lowercase letters $\bm$, $\bl$ and so on. Whenever some statement
refers to a sequence denoted by a lowercase letter such as $\bm$, it will be understood that we are dealing with a sequence of quotients (of the sequence $\M$ given by \eqref{eqMfromm}).
\end{rema}

The following properties are easy consequences of the definitions, except for (ii.2),
which is due to H.-J. Petzsche and D. Vogt~\cite[Lemma\ 5.3]{petvog}, and for (iii), which is due to H.-J. Petzsche~\cite[Prop.\ 1.1]{Pet}.

\begin{prop}\label{propPropiedlcmg}
Let $\M=(M_p)_{p\in\N_0}$ be a sequence. Then, we have:
\begin{itemize}
\item[(i)] $\M$ is (lc) if, and only if, $\bm$ is nondecreasing. If, moreover, $\M$ satisfies (snq) then $\bm$ tends to infinity.
\item[(ii)] Suppose that $\M$ is (lc). Then,
\begin{itemize}
\item[(ii.1)]  $(M_p^{1/p})_{p\in\N}$ is nondecreasing, and $M_p^{1/p}\leq m_{p-1}$ for every $p\in\N$.\par
    Moreover, $\lim_{p\to\infty}m_p=\infty$ if, and only if, $\lim_{p\to\infty}M_p^{1/p}=\infty$.

\item[(ii.2)] The following statements are equivalent:
\begin{itemize}
  \item[(ii.2.a)]  $\M$ is (mg),
  \item[(ii.2.b)] $\sup_{p\in\N} m_p/M^{1/p}_p <\infty$,
  \item[(ii.2.c)]  $\sup_{p\in\N_0} m_{2p}/m_{p}<\infty$,
  \item[(ii.2.d)] $\sup_{p\in\N} \left(M_{2p}/M_p^2\right)^{1/p}<\infty$.
 \end{itemize}
\item[(ii.3)] If  $\M$ is (mg) and $A>0$ is the corresponding constant, then
\begin{equation*}
m_p^p\le A^{2p}M_p,\quad p\in\N_0.
\end{equation*}
\end{itemize}
\item[(iii)] If $(p!M_p)_{p\in\N_0}$ is (lc), then
 the following statements are equivalent:
 \begin{itemize}
 \item[(iii.1)] $\M$ verifies (snq).
 \item[(iii.2)] There exists $k\in\N$, $k\ge 2$, such that
 \begin{equation*}
  \liminf_{p\ri\oo} \frac{m_{kp}}{m_p}>1.
 \end{equation*}
 \end{itemize}
\end{itemize}
\end{prop}

In the next definitions and results we take into account the conventions adopted in Remark~\ref{remaequivMm}.

\begin{defi} Let $\M=(M_p)_{p\in\N_0}$ and $\L=(L_p)_{p\in\N_0}$ be sequences, we say that {\it $\M$ is equivalent to $\L$},
and we write $\M\approx\L$, if there exist $C,D>0$ such that
 $$D^p L_p \leq M_p \leq C^p L_p, \qquad\, p \in \N_0.$$
\end{defi}

\begin{defi} Let $\bm=(m_p)_{p\in\N_0}$ and $\bl=(\ell_p)_{p\in\N_0}$ be sequences,
we say that {\it $\bm$ is equivalent to $\bl$}, and we write $\bm\simeq\bl$, if there exist $c,d>0$ such that
 $$d  \ell_p \leq m_p \leq c \ell_p, \qquad p \in \N_0.$$
\end{defi}

The following statements are straightforward.

\begin{prop}\label{propRelacionOrdenes}
Let $\M$ and $\L$ be sequences.
\begin{itemize}
\item[(i)] If  $\bm\simeq\bl$ then  $\M\approx\L$.
\item[(ii)] If $\M$ and $\L$ are (lc) and one of them is (mg), then $\M\approx\L$ amounts to $\bm\simeq\bl$.
In particular, for strongly regular sequences one may equally use $\simeq$ and $\approx$.
\end{itemize}
\end{prop}

\begin{exam}\label{exampleSequences}
We mention some interesting examples. In particular, those in (i) and (iii) appear in the applications of summability theory to the study of formal power series solutions for different kinds of equations.
\begin{itemize}
\item[(i)] The sequences $\M_{\a,\b}:=\big(p!^{\a}\prod_{m=0}^p\log^{\b}(e+m)\big)_{p\in\N_0}$, where $\a>0$ and $\b\in\R$, are strongly regular (in case $\b<0$, the first terms of the sequence have to be suitably modified in order to ensure (lc)). In case $\b=0$, we have the best known example of strongly regular sequence, $\M_{\a,0}=(p!^{\a})_{p\in\N_{0}}$, called the \textit{Gevrey sequence of order $\a$}.
\item[(ii)] The sequence $\M_{0,\b}:=(\prod_{m=0}^p\log^{\b}(e+m))_{p\in\N_0}$, with $\b>0$, is (lc), (mg) and $\bm$ tends to infinity, but (snq) is not satisfied.
\item[(iii)] For $q>1$, $\M_q:=(q^{p^2})_{p\in\N_0}$ is (lc) and (snq), but not (mg).
\end{itemize}
\end{exam}

\section{Sequences that define a nonzero proximate order}\label{sectSequencesDefineProxOrder}

Our first results, gathered in this section, characterize those sequences for which one can define, in a straightforward and natural way, a nonzero proximate order.

\subsection{Regularly varying sequences}

There exists a deep connection between the notions of proximate order and regular variation, as it may be seen in the classical work of N.H.~Bingham, C.M.~Goldie and J.L.~Teugels~\cite[Ch.\ 7]{bingGoldTeug}. This fact gave us the inspiration for our next results.
We will use the main statements about regularly varying sequences, taken from the paper of R. Bojanic and E. Seneta~\cite{BojanicSeneta}.

\begin{den}[\cite{BojanicSeneta}]
A sequence $(s_p)_{p\in\N_0}$ of positive numbers is \textit{regularly varying} if
\begin{equation}\label{equation.def.SRV}
\lim_{p\ri\oo} \frac{ s_{\lfloor \lambda p\rfloor}}{s_p}=\psi(\lambda) \in (0,\oo)
\end{equation}
 for every $\lambda>0.$
\end{den}

Regularly varying sequences admit a very convenient representation.

\begin{theo}[\cite{BojanicSeneta},\ Th.\ 1 and 3]\label{theo.BojanicSeneta}
If $(s_p)_{p\in\N}$ is a regularly varying sequence, there exists a real number $\o$ (called the \textit{index of regular variation} of $(s_p)_{p\in\N}$) such that $\psi(\lambda)=\lambda^\o$ (see~\eqref{equation.def.SRV}). Moreover, there exist sequences of positive numbers
$(C_p)_{p\in\N}$ and $(\delta_p)_{p\in\N}$, converging to $C\in(0,\oo)$ and zero, respectively, such that
$$s_p=p^{\o} C_p \exp  \left(\sum^{p}_{j=1}\delta_j/j\right),\quad p\in\N. $$
Conversely, such a representation for a sequence $(s_p)_{p\in\N}$ implies it is regularly varying of index $\o$.
\end{theo}

For convenience, given a sequence $\M$ of positive real numbers we define
\begin{equation}\label{equaDefiAlfaBeta}
\a_p:=\log(m_p),\quad p\in\N_0; \qquad \b_0:=\a_0, \qquad \b_p:=\log\left(\frac{m_p}{M^{1/p}_p}\right), \quad p\geq1.
\end{equation}

The following proposition, interesting in its own right, is a new characterization of regular variation for sequences in terms of the existence of a limit closely related
with the construction of proximate orders, as it will be shown in Theorem~\ref{theorem.charact.prox.order.nonzero}.

\begin{prop}\label{prop.beta.regularvariation.bis}
 Let $\M=(M_p)_{p\in\N_0}$ be a sequence of positive real numbers.
 The following are equivalent:
\begin{enumerate}[(i)]
 \item There exists $\lim_{p\to\infty} \log\big(m_p/M_p^{1/p}\big)\in\R$.
 \item $\m$ is regularly varying.
\end{enumerate}
In case any of these statements holds, the value of the limit in (i) and the index of regular variation of $\m$ are the same.
\end{prop}

\begin{proof} (i) $\Rightarrow$ (ii)  We call $\o$ the value of the limit in (i). If we consider the sequences
$(\a_p)_{p\in\N_0}$ and $(\b_p)_{p\in\N_0}$
defined in~\eqref{equaDefiAlfaBeta}, we will show that
$$\lim_{p\ri\oo}( \a_{\lfloor\lambda p\rfloor}-\a_p)=\o\log(\lambda),\qquad \lambda>0 ,$$
 which, by definition, implies condition (ii) and, moreover, by Theorem~\ref{theo.BojanicSeneta}, shows that the index of regular variation is equal to $\o$. For $\lambda=1$ the result is immediate. Assuming that $\lambda>1$, and using Lemma 3.8 in \cite{JimenezSanz} we know that
\begin{align}
  \a_p=&\sum^{p-1}_{k=0}\frac{\b_k}{k+1}+\b_p, \qquad p\in\N_0.\label{eqalfabeta}
 \end{align}
This leads to
\begin{equation}
 \a_{\lfloor\lambda p\rfloor}-\a_p=\sum^{\lfloor\lambda p\rfloor-1}_{k=p}\frac{\b_k}{k+1}+\b_{\lfloor\lambda p\rfloor}-\b_p.\no
\end{equation}
Condition~(i) can be written as
$\lim_{p\ri\oo}\b_p=\o$,
so it is sufficient to prove that
$$\lim_{p\ri\oo} \sum^{\lfloor\lambda p\rfloor-1}_{k=p}\frac{\b_k}{k+1} =\o\log(\lambda).$$
If we take $\ep>0$, we fix $\delta>0$ such that $\delta\log(\lambda)<\ep/6$.
There exists $p_\delta\in\N$ such that
$|\b_p-\o|<\delta$ for $p\geq p_\delta$.
We remember that the $p-$th partial sum $H_p=\sum_{k=1}^p 1/k$ of the harmonic series may be given as
\begin{equation}\label{equaPartialSumsHarmonic}
H_p=\log(p)+\ga+\ep_p,\qquad \ga=\textrm{Euler's constant},\ \lim_{p\ri\oo} \ep_p=0.
\end{equation}
Consequently, for $p\geq p_\delta$ we have
$$\sum^{\lfloor\lambda p\rfloor-1}_{k=p}\frac{\b_k}{k+1} \leq (\o+\delta) (H_{\lfloor\lambda p\rfloor}-H_p) =
(\o+\delta) \left(\log\left(\frac{\lfloor\lambda p\rfloor}{\lambda p}\right)+\log(\lambda) +\ep_{\lfloor\lambda p\rfloor}-\ep_p\right).$$
Using that $\lim_{p\ri\oo}\lfloor\lambda p\rfloor /(\lambda p) =1$ and that $\lim_{p\ri\oo} \ep_p=0$, we take $p_0\geq \max(p_\delta,(\lambda-1)^{-1})$ such
that for every $p\geq p_0$ one has
$$\left|\o\log\left(\frac{\lfloor\lambda p\rfloor}{\lambda p}\right) \right|<\ep/12, \quad
\left|\de\log\left(\frac{\lfloor\lambda p\rfloor}{\lambda p}\right) \right|<\ep/12,\quad |\o\ep_p|<\ep/6,
\quad |\de\ep_p|<\ep/6.$$
Then for $p\geq p_0$ we see that $\lfloor\lambda p\rfloor\ge p$, and so
$$\sum^{\lfloor\lambda p\rfloor-1}_{k=p}\frac{\b_k}{k+1} <\o\log(\lambda)+\ep.$$
Analogously, for $p\geq p_0$ we may also get that
$$ \o\log(\lambda)-\ep<\sum^{\lfloor\lambda p\rfloor-1}_{k=p}\frac{\b_k}{k+1},$$
and we are done.\par
For $\lambda\in(0,1)$, the proof is similar and we omit it.

(ii) $\Rightarrow$ (i) Let $\o$ be the index of regular variation of $\m$. By Theorem~\ref{theo.BojanicSeneta} one may write, after easy adaptations,
$$m_p=(p+1)^{\o} C_p \exp  \left(\sum^{p}_{k=1}\frac{\delta_k}{k}\right),\quad p\in\N_0, $$
where $(C_p)_{p\in\N_0}$ and $(\delta_p)_{p\in\N}$ are sequences of positive numbers converging to $C\in(0,\oo)$ and zero,
respectively. Then
\begin{align*}
 M_p &= m_0m_1\cdots m_{p-1}=p!^{\o} \left(\prod_{k=0}^{p-1} C_k \right)\exp  \left(\sum^{p-1}_{j=1} \sum^{j}_{k=1}\frac{\delta_k}{k}\right)\\
 & =
p!^{\o} \left(\prod_{k=0}^{p-1} C_k\right) \exp  \left(\sum^{p-1}_{j=1} (p-j)\frac{\delta_j}{j}\right) \\
&= p!^{\o} \left(\prod_{k=0}^{p-1} C_k\right) \exp  \left(p\sum^{p-1}_{j=1} \frac{\delta_j}{j} - \sum^{p-1}_{j=1} \delta_j \right).
\end{align*}
Using that $p!^{\o/p}\sim p^{\o} e^{-\o} (\sqrt{2\pi p})^{\o/p}$ and that
$\lim_{p\ri\oo} \left(\prod_{k=0}^{p-1} C_k\right)^{1/p}=C $,
we see that
$$\lim_{p\ri\oo} \frac{m_p}{M^{1/p}_p} =\lim_{p\ri\oo} e^{\o} (\sqrt{2\pi p})^{-\o/p} \exp  \left( \frac{1}{p} \sum^{p}_{j=1}\delta_j \right)=e^{\o}$$
or, equivalently, $\lim_{p\ri\oo} \b_p=\o$.
\end{proof}

We will also need the following theorem of L.~de~Haan~\cite{haan}, that shows that if we have monotonicity,
we only need to prove~(\ref{equation.def.SRV}) for two suitable integer values of $\lambda$. Then, if the sequence $\M$ is (lc),
we can give a nicer expression for the regular variation of $\m$.

\begin{theo}[\cite{haan},\ Th.\ 1.1.2] \label{theo.de.haan}
A positive monotone sequence $(s_p)_{p\in\N_0}$ varies regularly if there exist
positive integers $\ell_1,\ell_2$ with $\log(\ell_1)/\log(\ell_2)$ finite and irrational such that for some real number $\o$,
$$\lim_{p\ri\oo} \frac{s_{\ell_j p}}{s_p}=\ell_j^{\o}, \qquad j=1,2.$$
\end{theo}

\subsection{Characterizations for $d_{\M}$ being a nonzero proximate order}

If $\M=(M_p)_{p\in\N_0}$ is (lc) with $\lim_{p\to\infty} m_p=\infty$,
S. Mandelbrojt considers in \cite{mandelbrojt} its \textit{associated function} $M(t)$ given in~\eqref{eqDefinMdet},

which may be computed as
 \begin{equation*}\label{equaexprMdet}
M(t)=\left \{ \begin{matrix}  p\log t -\log(M_{p}) & \mbox{if }t\in [m_{p-1},m_{p}),\ p=1,2,\ldots,\\
\\
0 & \mbox{if } t\in [0,m_{0}). \end{matrix}\right.
\end{equation*}
For later use, we note that
\begin{equation}
 M(m_p)=\log \left( \frac{m^p_p}{M_p}\right), \qquad p\in \N_0. \label{equation.M.in.mp}
\end{equation}
We can also consider the \textit{counting function} $\nu:(0,\infty)\to\N_0$ for the sequence of quotients $\bm$, given by
\begin{equation*}
\nu(t)=\#\{j:m_j\le t\},
\end{equation*}
which allows one to write
\begin{equation}\label{equaRelationM_nu}
M(t)=\nu(t)\log(t)-\log(M_{\nu(t)}),\quad t>0;\ M'(t)=\frac{\nu(t)}{t},\quad t>0,\ t\neq m_p,\ p\in\N_0.
\end{equation}

The link between proximate orders and sequences is given by the function
\begin{equation}\label{equaDefidM}
d_{\M}(t)=\frac{\log(M(t))}{\log (t)},\quad t\textrm{ large enough}.
\end{equation}
Based on a theorem of S. Mandelbrojt, we have the following theorem relating the function $d_{\M}$, the sequence $(m_p)_{p\in\N_0}$ and the index $\o(\M)$ given by
\begin{equation*}
\omega(\M)=\liminf_{p\to\infty}
\frac{\log(m_{p})}{\log(p)},
\end{equation*}
which for a (lc) sequence $\M=(M_p)_{p\in\N_0}$ with $\lim_{p\to\infty} m_p=\infty$ can be 0, a positive real number or $\infty$. This index plays a prominent role in the study of quasianalyticity in Carleman ultraholomorphic classes, see~\cite{JimenezSanz} and~\cite{SanzAsymptoticAnalysis}.

\begin{theo}[\cite{JimenezSanz}, Th.\ 3.2;\ \cite{SanzAsymptoticAnalysis}, Th.\ 2.24 and Th.\ 4.6 ]\label{theo.limsup.dM.omega}
Let $\M$  be (lc) with $\lim_{p\ri\oo}m_p=\oo$, then
$$\limsup_{t\to\infty} d_{\M}(t)
=\limsup_{p\to\infty}\frac{\log(p)}{\log(m_{p})}=\frac{1}{\o(\M)}$$
(where the last quotient is understood as 0 if $\o(\M)=\infty$, and as $\infty$ if $\o(\M)=0$).
\end{theo}

The main result of this section characterizes those sequences $\M$ for which $d_{\M}$ is a nonzero proximate order.

\begin{theo}\label{theorem.charact.prox.order.nonzero}
Let $\M=(M_p)_{p\in\N_0}$ be a (lc) sequence with $\lim_{p\to\infty} m_p=\infty$. The following are equivalent:
\begin{enumerate}[(a)]
 \item $d_{\M}(t)$ is a proximate order with $\lim_{t\to\infty}d_{\M}(t)\in(0,\infty)$.
 \item There exists $\lim_{p\to\infty} \log\big(m_p/M_p^{1/p}\big)\in(0,\infty)$.
 \item $\m$ is regularly varying with a positive index of regular variation.
 \item There exists $\o>0$ such that for every natural number $\ell\geq 2$,
 $$\lim_{p\ri\oo} \frac{m_{\ell p}}{m_p}=\ell^\o.$$
\end{enumerate}
In case any of these statements holds, the value of the limit mentioned in (b), that of the index mentioned in (c), and that of the constant $\o$ in (d) is $\o(\M)$,
and the limit in (a) is $1/\o(\M)$.
\end{theo}

\begin{proof}
For the sake of completeness, we will show more implications than strictly needed.\par
 (a) $\Rightarrow$ (b) According to \eqref{equaDefidM} and \eqref{equaRelationM_nu}

 we have that
$$d_{\M}'(t)=\frac{M'(t)}{\log(t)M(t)}-\frac{d_{\M}(t)}{t \log(t)}=
\frac{1}{t\log(t)}\left(\frac{\nu(t)}{M(t)}-d_{\M}(t)\right),$$
whenever it exists.
Observe that (\ref{OA4:1}) in Definition~\ref{OAD:1} amounts then to
\begin{equation}\label{equaRephraseConditionDProxOrder}
\lim_{t\ri \oo} \Big(\frac{\nu(t)}{M(t)} - d_{\M}(t)\Big)=0.
\end{equation}
By Theorem~\ref{theo.limsup.dM.omega} and condition (C) in Definition~\ref{OAD:1}, we know that $\lim_{t\ri\oo} d_{\M}(t) = 1/\o(\M)$, and so
\begin{equation*}
\lim_{t\ri \oo} \frac{\nu(t)}{M(t)}= \frac{1}{\o(\M)}.
\end{equation*}
In particular,
$$
\lim_{p\ri \oo} \frac{\nu(m_p)}{M(m_p)}=\lim_{p\ri \oo} \frac{p+1}{M(m_p)}=\frac{1}{\o(\M)}.
$$
Taking into account~(\ref{equation.M.in.mp}), the last limit may be written as
 \begin{equation*}
  \lim_{p\ri\oo} \log\left(\frac{m_p}{M_p^{1/p}}\right)=\o(\M),
 \end{equation*}
as desired.

(b) $\Rightarrow$ (a) According to~\eqref{equaDefiAlfaBeta}, condition~(b) can be written as
\begin{equation}\label{equaLimitBetaOmega}
\lim_{p\ri\oo} \b_p=\o\in(0,\oo).
\end{equation}
By using~(\ref{eqalfabeta}), we see that
\begin{equation*}
\lim_{p\ri\oo} \frac{(\a_{p+1}-\b_{p+1})-(\a_p-\b_p)}{\log(p+1)-\log(p)}=\lim_{p\ri\oo} \frac{\b_p/(p+1)}{1/p}=\o,
\end{equation*}
and then we deduce by Stolz's criterion that
\begin{equation*}
\lim_{p\ri\oo} \frac{\a_p-\b_p}{\log(p)}=\o.
\end{equation*}
Since $\beta_p=O(1)$ (and $\a_p=\log(m_p)$), we get
\begin{equation}\label{equaAuxilLimit1}
\lim_{p \ri \oo} \frac{\log(m_p)}{\log(p)}=\o.
\end{equation}
On the other hand, there exist $a,A>0$ and $p_0\in\N$ such that
$$ a<\log\big(m_p/M_p^{1/p}\big)<A,\qquad p\geq p_0,$$
what, by~(\ref{equation.M.in.mp}) and taking logarithms, amounts to
$$ \log(a)+\log(p)<\log(M(m_p))<\log(p)+\log(A),\qquad p\geq p_0.$$
Consequently, we see that
\begin{equation}\label{equaAuxilLimit2}
\lim_{p \ri \oo} \frac{\log (M(m_p))}{\log (p)} =1.
\end{equation}
Observe that $M(t)$ is nondecreasing, so for every $t\in(m_{p-1},m_p)$ we have
\begin{align*}
\frac{p}{M(m_p)}&\leq\frac{\nu(t)}{M(t)}\leq \frac{p}{M(m_{p-1})},\\
\frac{\log(M(m_{p-1}))}{\log(m_{p})}&\leq \frac{\log(M(t))}{\log(t)}\leq \frac{\log(M(m_p))}{\log(m_{p-1})}.
\end{align*}
By~(\ref{equation.M.in.mp}) we know that $M(m_p)=p\b_p$ for every $p\in\N$,
so from \eqref{equaLimitBetaOmega} and the first inequalities we see that
$\lim_{t\to\oo}\nu(t)/M(t)=1/\o$.
Now, using~\eqref{equaAuxilLimit1} and \eqref{equaAuxilLimit2} we conclude from the second inequalities that
$\lim_{t \ri \oo} d_{\M}(t)=1/\o$, and also that \eqref{equaRephraseConditionDProxOrder} is satisfied. So, (\ref{OA3:1}) and (\ref{OA4:1}) in Definition~\ref{OAD:1} are valid and $d_{\M}$ is a proximate order. Moreover, by Theorem~\ref{theo.limsup.dM.omega} we deduce that $\o=\o(\M)$.

 (b) $\Leftrightarrow$ (c) Apply Proposition~\ref{prop.beta.regularvariation.bis}.

 (c) $\Rightarrow$ (d) By definition, we know that for every natural number $\ell\ge 2$ the limit
 $$
 \lim_{p\to\infty}\frac{m_{\ell p}}{m_p}=\psi(\ell)\in(0,\infty)
 $$
 exists, and by Theorem~\ref{theo.BojanicSeneta} it equals $\ell^{\o}$, where $\o$ is the index of regular variation of $\m$.

 (d) $\Rightarrow$ (c) Since $\m$ is nondecreasing, it suffices to apply Theorem~\ref{theo.de.haan}.

The value of the different limits or indices involved in the statements is deduced in the course of the proof.
\end{proof}

\begin{exam}\label{exampleOmega}
For the sequences in the Example~\ref{exampleSequences} we easily have that $\o(\M_{\a,\b})=\a$  and $\o(\M_q)=\infty$ for the considered values of $\a$, $\b$ and $q$. So, Theorem~\ref{theorem.charact.prox.order.nonzero} shows that for the sequences $\M_{0,\b}$ and $\M_q$ the function $d_{\M}$ is not a nonzero proximate order. On the contrary, one may easily check that (b) or (d) in that Theorem hold for $\M_{\a,\b}$ whenever $\a>0$, and consequently $d_{\M_{\a,\b}}$ is indeed a nonzero proximate order, although its handling will be difficult in general (in this sense, see Remark~\ref{remaEasyProxOrder}).
\end{exam}

We will show next some necessary conditions for $d_{\M}$ being a nonzero proximate order. For that purpose, we recall the definition of the growth index $\gamma(\M)$ given by V. Thilliez.

\begin{defi}[\cite{thilliez}]
Let $\bM=(M_{p})_{p\in\N_{0}}$ be a strongly regular sequence and $\ga>0$. We say $\bM$ satisfies property
$\left(P_{\ga}\right)$  if there exist a sequence of real numbers $m'=(m'_{p})_{p\in\N_0}$ such that $m'\simeq\m$ and $\left((p+1)^{-\ga}m'_{p}\right)_{p\in\N_0}$
is increasing.
The \textit{growth index} of $\bM$ is
$$\ga(\bM):=\sup\{\ga\in\R:(P_{\ga})\hbox{ is fulfilled}\}\in(0,\infty).$$
\end{defi}

\begin{rema}\label{remaAlmostIncrease}
In~\cite[Prop.\ 4.12]{JimenezSanz}, a link was given between this index and the property of almost increase. We recall that
a sequence $(s_p)_{p\in\N_0}$ of positive real numbers is \textit{almost increasing} if
there exists a constant $M\geq1$ such that
\begin{equation}
s_p\leq M s_{q}, \qquad \text{for every} \quad p,q \in \N_0,\ p\le q.\no
\end{equation}
Then, $\bM$ satisfies property
$\left(P_{\ga}\right)$  if, and only if, the sequence $\left((p+1)^{-\ga}m_{p}\right)_{p\in\N_0}$ is almost increasing.
\end{rema}

We can state now our last result in this section.

\begin{coro}\label{coroConseqRegularVariation}
Let $\M=(M_p)_{p\in\N_0}$ be a (lc) sequence with $\lim_{p\to\infty} m_p=\infty$, and verifying any of the conditions in the previous result.
Then, the following holds:
\begin{enumerate}
\item[(g)]  $\M$ is strongly regular, $\gamma(\M)=\o(\M)$ and
\begin{equation}\label{limit.logmp.over.logp.omega}
\lim_{p\to\infty} \frac{\log(m_p)}{\log(p)}=\o(\M).
\end{equation}
\end{enumerate}
\end{coro}

\begin{proof}
Since $\M$ verifies (d) in Theorem~\ref{theorem.charact.prox.order.nonzero} for $\ell=2$, it is clear that $\M$ satisfies the conditions in Proposition~\ref{propPropiedlcmg}, items (ii.2.c) and (iii.2), and so $\M$ is also (mg) and (snq) (note that $(p!M_p)_{p\in\N_0}$ is (lc), being the product of two sequences which are (lc)).

The equality $\gamma(\M)=\o(\M)$ admits the same proof as in Theorem 4.19 in~\cite{JimenezSanz}, once we depart from the regular variation of $\m$ with index $\o(\M)$ (by Theorem~\ref{theorem.charact.prox.order.nonzero}(c)).

Finally, the equality~\eqref{limit.logmp.over.logp.omega} has been deduced in the proof that (b) $\Rightarrow$ (a) in the previous theorem.
\end{proof}

\begin{rema}\label{remaCorrections}
As explained in the introduction, some statements for strongly regular sequences in the paper~\cite{JimenezSanz} by the first two authors are not correct. The forthcoming Example~\ref{example.dM.not.p.o} shows that the implication (iii) $\Rightarrow$ (iv) in Theorem~\ref{theorem.condition3.implies.beta.converge} fails in general, while the converse, as shown here, is valid. Indeed, for $\M$ strongly regular one has (i) $\Leftrightarrow$ (iv) and (ii) $\Leftrightarrow$ (iii). It turns out that the condition (iii), which was involved in the statements of Proposition 4.6, Theorem 4.10 and Theorem 4.19 in~\cite{JimenezSanz}, should be substituted throughout by condition (iv) or, equivalently, by the regular variation of $\m$ with positive index. Moreover, Remarks 3.15 and 3.16 in~\cite{JimenezSanz} become meaningless.
\end{rema}

\begin{exam}\label{example.dM.not.p.o}
Let $\M$ be defined using the sequence of  quotients
$(m_p)_{p\in \N_0}$. We put $m_0=m_1=1$, $m_2=m_3=2$ and $m_4=m_5=m_6=m_7=6$; for every $k\in\N$ and $2^{2^k+1}\leq p <2^{2^{k+1}+1}$ we
define $m_p$ as follows:
$$m_p= 2^{2^k} 3\left(\frac{2^{2^k}}{3}\right)^{\frac{j-1}{2^k-1}},  \qquad 2^{2^k+j}\le p\le 2^{2^k+j+1}-1,\quad
j=1,2,\dots, 2^k.$$
Since

$m_8=12$, in order to obtain the property (lc) we need to show
that $(m_p)_{p\geq 8}$ is nondecreasing. For every $k\in\N$
there are three possibilities:
\begin{enumerate}
 \item If $p,p+1\in[2^{2^k+j},2^{2^k+j+1}-1 ]$ for $j=1,\dots,2^k$, we have that $m_{p+1}/m_p=1$.
 \item If $p=2^{2^k+j+1}-1$ for $j=1,\dots,2^k-1$,  we have that $m_{p+1}/m_p=(2^{2^k}/3)^{1/(2^k-1)}$, which is greater
 than $1$ since $k\in\N$.
 \item If $p=2^{2^{k+1}+1}-1$, we have that $m_{p+1}/m_p= 2^{2^{k+1}} 3/ 2^{2^{k+1}} =3$.
\end{enumerate}

Next we analyze the quotients $m_{2p}/m_p$. By definition, for any  $p\in [2^{2^k+j},2^{2^k+j+1}-1]$
we have that $2p$ belongs to the adjacent interval $[2^{2^k+j+1},2^{2^k+j+2}-1]$. We distinguish two cases:
\begin{enumerate}
 \item If $p \in[2^{2^k},2^{2^k+1}-1 ]$ we have that $m_{2p}/m_p= 3$.
 \item If $p\in [2^{2^k+j},2^{2^k+j+1}-1 ] $ for $j=1,\dots,2^k-1$,  we have that
 $m_{2p}/m_p=(2^{2^k}/3)^{1/(2^k-1)}$.
\end{enumerate}
We observe that
 $$\lim_{k\ri\oo} \frac{2^{2^k/(2^k-1)}}{3^{1/(2^k-1)}}=2.$$
From both cases, we have that
$$ 1<2=\liminf_{p\ri\oo}\frac{m_{2p}}{m_p} \leq \limsup_{p\ri\oo}\frac{m_{2p}}{m_p}=3<\oo.$$
Using Proposition~\ref{propPropiedlcmg}, items (ii.2) and (iii), we see that $\M$ is (mg) and (snq). However, the limit
$$\lim_{p\ri\oo}\frac{m_{2p}}{m_p}$$
does not exist, then condition (d) in Theorem~\ref{theorem.charact.prox.order.nonzero} is violated and
$d_{\M}(t)$ is not a nonzero proximate order.

Next, we are going to see that $\bm\simeq\bl$, where the sequence $\bl=(\ell_p)_{p\in\N_0}$, with $\ell_p=p+1$ for every $p\in\N_0$, corresponds to the Gevrey sequence of order 1. For every $p\geq8$ and
$2^{2^k+j}\le p\le 2^{2^k+j+1}-1$, we have that
$$\frac{2^{2^k} 3\left( 2^{2^k}/3\right)^{\frac{j-1}{2^k-1}}}{2^{2^k+j+1}} \leq\frac{m_p}{p}\leq \frac{2^{2^k} 3\left(2^{2^k}/3\right)^{\frac{j-1}{2^k-1}}}{2^{2^k+j}}.$$
Then
$$ 3^{\frac{2^k-j}{2^k-1}}  2^{\frac{j-2^k}{2^k-1}-1} \leq\frac{m_p}{p}\leq 3^{\frac{2^k-j}{2^k-1}}  2^{\frac{j-2^k}{2^k-1}}. $$
Since $j=1,2,\dots, 2^k$, we see that
$$  2^{-2} \leq\frac{m_p}{p}\leq 3, $$
from where the equivalence is clear. Since the existence, and the value if it exists, of the limit appearing in~(\ref{limit.logmp.over.logp.omega}) are stable under equivalence, and moreover it is obvious that $\lim_{p\to\oo}\log(\ell_p)/\log(p)=1$, we deduce that $\m$  satisfies~(\ref{limit.logmp.over.logp.omega}) with $\o(\M)=1$.
\end{exam}

\section{Sequences admitting a nonzero proximate order}\label{sectSequencesAdmitProxOrder}

Example~\ref{example.dM.not.p.o} provides a strongly regular sequence $\M$ such that $d_{\M}(t)$ is not a proximate order. However, this sequence is equivalent to $\L=(p!)_{p\in\N_0}$, and, as indicated in Example~\ref{exampleOmega}, $d_{\L}$ is  a nonzero proximate order (in particular, we deduce that the property of $d_{\M}$ being a proximate order is not stable under equivalence of sequences). So, we may obtain a satisfactory summability theory in the Carleman ultraholomorphic class associated to $\L$, which coincides with that associated to $\M$. This shows that asking for $d_{\M}$ to be a nonzero proximate order is too strong a restriction for a (lc) sequence $\M$ with $\bm$ tending to infinity, and one could ask instead for:

\medskip
(e) There exists a (lc) sequence $\L$, with quotients tending to infinity, such that $\L\approx\M$ and $d_{\L}(t)$ is a nonzero proximate order.

\medskip

On the other hand, the second author had already observed~\cite[Remark\ 4.11(iii)]{SanzFlat} that, for the construction of nontrivial flat functions in optimal sectors, $d_{\M}$ need not be a nonzero proximate order, but it is enough that there exist nonzero proximate orders close enough to $d_{\M}$, in the following sense.

\begin{defi}\label{def.admit.p.o}
Let $\M=(M_p)_{p\in\N_0}$ be a (lc) sequence with $\lim_{p\to\infty} m_p=\infty$. We say that $\M$ \textit{admits a proximate order} if there exist a proximate order $\ro(t)$
 and constants $A$ and $B$ such that
\begin{equation}\label{eqAdmitsProximateOrder}
A\leq \log(t)(\ro(t)-d_{\M}(t)) \leq B, \qquad t\textrm{ large enough}.
\end{equation}
\end{defi}

Observe that, if $d_{\M}(t)$ is a proximate order and $\rho(t)$ is another proximate order equivalent to $d_{\M}(t)$, then one has
$$
\lim_{t\to\infty}\log(t)(\ro(t)-d_{\M}(t))=0,
$$
and so \eqref{eqAdmitsProximateOrder} is verified.

In case $\M$ admits a proximate order, $d_{\M}$ verifies all the properties of proximate orders except possibly \eqref{OA4:1}, since it is clear from the definition of admissibility that
 $\lim_{t\to\oo}d_{\M}(t)=\lim_{t\to\oo}\ro(t)$ exists.

\begin{rema}\label{remaEasyProxOrder}
The admissibility condition is interesting even if  $d_{\M}$ is a proximate order. For example, consider the sequence $\M_{\a,\b}$ in Example~\ref{exampleSequences}, with $\a>0$, and let us put $M_{\a,\b}(t)$, $d_{\a,\b}(t)$, and so on, to denote the corresponding associated functions. Since for large $t$ we have
$$
c_2t^{1/\a}\log^{-\b/\a}(t)\le M_{\a,\b}(t)\le c_1t^{1/\a}\log^{-\b/\a}(t)
$$
for suitable constants $c_1,c_2>0$ (see~\cite[Example\ 1.2.2]{thilliez2}), then $$\log(c_2)\le\log(t)(d_{\a,\b}(t)-\ro_{\a,\b}(t))\le\log(c_1) \quad\textrm{eventually}
 $$
(see Example~\ref{exampleOrders} for the definition of $\ro_{\a,\b}$). This shows that the proximate order $\ro_{\a,\b}(t)$ is admissible for $\M_{\a,\b}$, and therefore, for our purposes, it may substitute $d_{\a,\b}(t)$ whenever it is convenient. In particular, when working with Gevrey ultraholomorphic classes one may consider the constant order $\ro_{\a,0}(t)\equiv 1/\a$, as expected.
\end{rema}

In order to show that the requirement (e) and the admission of a nonzero proximate order are equivalent for a sequence $\M$, we need to construct well-behaved sequences from proximate
orders.

\subsection{Strongly regular sequences defined from nonzero proximate orders}

Departing from a nonzero proximate order,
and for every element $V$ in the class $\Bet(\ga,\ro(r))$
given by L.S.~Maergoiz~\cite{Maergoiz} (see Theorem~\ref{theorpropanalproxorde} and Definition~\ref{defClassMaergoiz}), we will construct a well-behaved sequence $(M_p^V)_{p\in\N_0}$.
This procedure uses the same
argument by S.~Mandelbrojt~\cite{mandelbrojt} and H.~Komatsu~\cite{komatsu} to recover  a sequence from
 its associated function $M(t)$, or by J.~Bonet, R.~Meise and S.N.~Melikhov~\cite{bonetmeisemelikhov} when they construct a weight sequence from a weight function.

\begin{defi}
Let $\ro(t) \ri \ro>0$ be a proximate order,  $\ga>0$ and $V \in \Bet(\ga, \ro(t) )$. We define its associated sequence by
$$
M^V_p:=\sup_{t>0} \frac{t^p}{e^{V(t)}},\qquad p\in\N_0,
$$
or equivalently,
$$
\log(M^V_p)=\sup_{t>0} (p\log(t)-V(t)),\qquad p\in\N_0.
$$
\end{defi}

Note that, since $\rho>0$, for every $p\in\N_0$ we have that $\lim_{t\ri\oo} t^p/e^{V(t)}=0$. As $e^{-V(1)}>0$, we see  that $M^V_p \in (0,\oo)$.

The forthcoming results aim at showing, in Theorem~\ref{theorem.construct.sequences.from.prox.order}, that the sequence $(M_p^V)_{p\in\N_0}$ is strongly regular. We will frequently use the properties, denoted by (I)-(VI) in Theorem~\ref{theorpropanalproxorde}, of the elements in the class $\Bet(\ga, \ro(t) )$.

  \begin{prop}\label{MVlogconvex}
  Let $\ro(t) \ri \ro>0$ be a proximate order,  $\ga>0$ and $V \in \Bet(\ga, \ro(t) )$, then the
  sequence $(M^V_p)_{p\geq0}$ is logarithmically convex and $M^V_0=1$.
 \end{prop}

 \begin{proof}
  For any $p\in\N$ we have that
  $$(M^V_p)^2=\sup_{t>0} \frac{t^{2p}}{e^{2V(t)}} = \sup_{t>0} \frac{t^{p-1+p+1}}{e^{2V(t)}} \leq \sup_{t>0} \frac{t^{p-1}}{e^{V(t)}}
  \sup_{t>0} \frac{t^{p+1}}{e^{V(t)}} = M^V_{p-1}M^V_{p+1}.$$
 Using (III), $V(t)$ is strictly increasing and $\lim_{t\ri0}V(t)=0$, and so
  $$M^V_0=\sup_{t>0} e^{-V(t)} = \exp\left(\lim_{t\ri0}V(t)\right)=1. $$
 \end{proof}

 \begin{lemma}\label{prop.properties.V.A}
  Let $\ro(t) \ri \ro>0$ be a proximate order,  $\ga>0$ and $V \in \Bet(\ga, \ro(t) )$. We consider the function
  $$A(s):=\ro_V'(s)s\log(s)+\ro_V(s),$$
where $\rho_V(s)=\log(V(s))/\log(s)$ (see Theorem~\ref{theorpropanalproxorde}(VI)).  Then the function $V(s)A(s)$ has the following properties:
  \begin{enumerate}[(1)]
  \item $\lim_{s\ri\oo} V(s)A(s)=\oo$,
  \item $V(s)A(s)\geq 0$ in (0,$\oo$),
 \item $V(s)A(s)$ is strictly increasing in (0,$\oo$),
 \item $\lim_{s\ri0}V(s)A(s)=L\geq 0.$
\end{enumerate}

 \end{lemma}

 \begin{proof}

By properties (C) and (D) of proximate orders,
 $$A(s)=\ro_V'(s)s\log(s)+\ro_V(s)\ri \ro, \qquad s\ri\oo.$$
Consequently, $\lim_{s\ri\oo} V(s)A(s)=\oo. $
By property (III) of functions in $\Bet(\ga, \ro(t))$, $V(s)$  is strictly increasing, so
\begin{equation*}
V'(s)= \frac{V(s)}{s}[\ro_V'(s)s\log(s)+\ro_V(s)] = \frac{V(s)}{s} A(s)\geq 0
\end{equation*}
for every $s\in(0,\oo)$. Consequently, as $V(s)$ is positive, $A(s)$ is nonnegative, and their product is also nonnegative.
We also see, by property (IV) of functions in $\Bet(\ga, \ro(t))$, that
$V(e^t)$ is strictly convex in $\R$, so $(V(e^t))'$ is strictly increasing in $\R$. We see that
$$ (V(e^t))'= (\exp(\ro_V(e^t) t))' = \exp(\ro_V(e^t)t)[\ro_V'(e^t) t e^t+\ro_V(e^t) ] = V(e^t)A(e^t) $$
for every $t\in\R$. Making the change $e^t=s$, we deduce that the function $V(s) A(s)$ is strictly increasing in
$(0,\oo)$. Using that $V(s)A(s)$ is nonnegative and strictly increasing  we assure that
$$\lim_{s\ri0}V(s)A(s)=:L\geq 0.$$
 \end{proof}

 \begin{lemma}\label{lemmafuncg}
    Let $\ro(t) \ri \ro>0$ be a proximate order,  $\ga>0$ and $V \in \Bet(\ga, \ro(t) )$. For every $p\in\N$
    we consider the function defined in $(0,\oo)$ by
    $$g_{p}(s)= V(s)- p\log(s).$$
    For $p$ large enough, $g_{p}$ reaches its minimum value in a point $s_p\in(0,\oo)$. The
    sequence $(s_p)_{p\geq p_0}$ is increasing and
    $\lim_{p\ri\oo} s_p=\oo$.
 \end{lemma}

 \begin{proof}
 Using that $\lim_{s\ri0} V(s)=0$ we see that $\lim_{s\ri 0} g_{p}(s)=\oo$.
By (VI), $\ro_V(s)$ is a proximate order equivalent to $\ro(s)$, and by Remark~\ref{remaEquivProxOrder}, $\lim_{s\ri\oo} \ro_V(s)=\ro>0$, so we have that
 $$\lim_ {s\ri\oo} g_{p}(s) =\oo.$$
 Then, $g_{p}$ reaches its minimum at a point $s_p$ in $(0,\oo)$. Let us see that it is unique for $p$ large enough. We calculate
 $$g'_{p}(s)=-\frac{p}{s}+(s^{\ro_V(s)})'= -\frac{p}{s}+ V(s) [\ro_V'(s)\log(s)+\frac{\ro_V(s)}{s} ],$$
which vanishes if, and only if, $ V(s) A(s)=p$.
By the properties in Lemma~\ref{prop.properties.V.A},
given $p\in\N$ large enough, there exists only one point $s_p\in(0,\oo)$ verifying $V(s_p)A(s_p)=p$. So $g_{p}(s)$
has a unique minimum in $s_p\in (0,\oo)$.

Furthermore, as $V(s)A(s)$ is strictly increasing and $\lim_{s\ri\oo} V(s)A(s)=\oo$, we deduce that the sequence $(s_p)_{p\geq p_0}$
is strictly increasing and $\lim_{p\ri\oo} s_p=\oo$.
 \end{proof}

  \begin{prop}\label{theorelMandU}
  Let $\ro(t) \ri \ro>0$ be a proximate order,  $\ga>0$ and $V \in \Bet(\ga, \ro(t) )$. Then, there exists $B>1$ such that
   \begin{equation}\label{equation.weak.UyM}
   \left(\frac{1}{B}\right)^p U(p)^p\leq M^V_p\leq B^p U(p)^p, \qquad p \in \N,
   \end{equation}
  where $U(s)$ is the inverse of $V(t)$ given in Remark~\ref{remaInverseRegularlyVarying}.
 \end{prop}
 \begin{proof}
 For simplicity we write $(M_p)_{p\geq0}$ for the sequence $(M^V_p)_{p\geq0}$.
 We consider the function $h_p(t)=t^p e^{-V(t)}$, and observe that
 $$g_{p}(t)=p\log \left(\frac{1}{t}\right)+V(t)=\log \left(\frac{1}{h_p(t)}\right).$$
 So $h_p(t)=\exp(-g_{p}(t))$, and as in Lemma~\ref{lemmafuncg} we deduce that
 $$\lim_{t\ri0}h_p(t) = 0, \qquad\lim_{t\ri\oo}h_p(t) = 0.$$
  Furthermore, for $p$ large enough, $h_{p}$ reaches its maximum value in a unique point $t_p\in(0,\oo)$, the sequence $(t_p)_{p\geq p_0}$ is increasing,
    $\lim_{p\ri\oo} t_p=\oo$ and $p=V(t_p)A(t_p)$.

  As $A(t)\ri\ro$ as $t\ri\oo$,
  if we fix $0<\ep<\ro$, there exists $K>0$ such that
  $$A(t)\in(\ro-\ep,\ro+\ep),\qquad t\geq K.$$
There exists a natural number $p_{K}$ such that for $p\geq p_{K}$ we have $t_p>K$.

Denote by $U$ the inverse function of $V$. Since $V$ is strictly increasing in $(0,\infty)$, the function $U$ is also defined in $(0,\infty)$.
As $p/A(t_p)\geq p/(\ro+\ep)$ for $p\geq p_K$, and
$V(t_p)A(t_p)=p$, we see that
$$U\left(\frac{p}{A(t_p)}\right)=t_p, \qquad p\geq p_K.$$
By the definition of $M_p$, for $p\geq p_K$ we observe that
$$M_p=\sup_{t>0} \frac{t^p}{e^{V(t)}}=\sup_{t>0} h_p(t)=h_p(t_p)= \frac{\left(U(p/A(t_p))\right)^p}{e^{V[U(p/A(t_p)]}}= \frac{\left(U(p/A(t_p))\right)^p}{e^{p/A(t_p)}}.$$
Using that the function $U(s)$ is increasing and that $A(t_p)\in(\ro-\ep,\ro+\ep) $,
we see that
$$\frac{\left(U(p/(\ro+\ep))\right)^p}{e^{p/(\ro-\ep)}}\leq M_p\leq \frac{\left(U(p/(\ro-\ep))\right)^p}{e^{p/(\ro+\ep)}},
\qquad p\geq p_K.$$
Finally, recall that $U$ is regularly varying (see Remark~\ref{remaInverseRegularlyVarying}), that is,
$$\lim_{s\ri\oo}\frac{U(rs)}{U(s)}=r^{1/\ro}$$
uniformly in the compact sets of $(0,\oo)$. So, considering the interval $[1/(\ro+\ep),1/(\ro-\ep)]$, we can assure that there exist
$P\in\N$, $P\geq p_K$, and $B>1$ such that
 $$\left(\frac{1}{B}\right)^p U(p)^p\leq M_p\leq B^p U(p)^p, \qquad p \in \N, \qquad p>P.$$
We conclude by suitably enlarging the constant $B$.
\end{proof}

\begin{theo}\label{theorem.construct.sequences.from.prox.order}
  Let $\ro(t) \ri \ro>0$ be a proximate order,  $\ga>0$ and $V \in \Bet(\ga, \ro(t) )$. Then, the
  sequence $(M^V_p)_{p\geq0}$ is strongly regular.
 \end{theo}

 \begin{proof}
 We write again $(M_p)_{p\geq0}$ for the sequence $(M^V_p)_{p\geq0}$.
 In Proposition~\ref{MVlogconvex} we have seen that  $(M_p)_{p\geq0}$ is logarithmically convex and $M_0=1$.
 By Proposition~\ref{propPropiedlcmg}(ii.2), in order to prove (mg) it is enough to see that there exist $p_0\in\N$ and $A>1$ such that
 $$M_{2p}\leq A^{p} M^2_p, \qquad p\geq p_0.$$
 We observe that
 $$M_{2p}= \sup_{t>0} \frac{t^{2p}}{e^{V(t)}}\leq \sup_{t>0} \frac{t^{p}}{e^{V(t)/2}} \sup_{t>0} \frac{t^{p}}{e^{V(t)/2}}.$$
We consider the functions $\eta_p(t)=t^p e^{-V(t)/2}$, and we see that the points that verify $\eta'_p(t)=0$ are
 those satisfying
 $$p=\frac{V(t)}{2}A(t).$$
Proceeding as in Proposition~\ref{theorelMandU}, and with the notation used there, we see that $\eta_p(t)$ reaches its maximum in the point $t_{2p}\in(0,\oo)$, and given
$0<\ep<\ro$ there exists $p_\ep\in\N$ such that for $p>p_\ep$ we have
$$A(t_{2p})\in (\ro-\ep,\ro+\ep), \qquad t_{2p}= U\left(\frac{2p}{A(t_{2p})}\right),$$
where $U(r)$ is the inverse of $V(s)$. So, for $p>p_\ep$ we have
 $$M_{2p}\leq  \frac{U\left(2p/A(t_{2p})\right)^{p}}{e^{p/A(t_{2p})}}  \frac{U\left(2p/A(t_{2p})\right)^{p}}{e^{p/A(t_{2p})}}.$$
Since $U$ is increasing, we see that
  $$M_{2p}\leq  \frac{U\left(2p/(\ro-\ep)\right)^{2p}}{e^{2p/(\ro+\ep)}},
  \qquad p >p_\ep.$$
 Finally, using that
$$\lim_{s\ri\oo}\frac{U(2s/(\ro-\ep))}{U(s)}=\left(\frac{2}{(\ro-\ep)}\right)^{1/\ro},$$
 there exists $p_0>p_\ep$ such that for every $p\geq p_0$ we have
 $$M_{2p}\leq  \frac{U(p)^{2p}2^{2p}\left(2/(\ro-\ep)\right)^{2p/\ro}}{e^{(2p)/(\ro+\ep)}},$$
 and by~(\ref{equation.weak.UyM})
  $$M_{2p}\leq  \left(\frac{4^{1+1/\ro}B^2}{e^{2/(\ro+\ep)}(\ro-\ep)^{2/\ro}}\right)^{p}M^2_p, \qquad p >p_0.$$
 Consequently, $\M$ has moderate growth.

Let us see that $\M$ also verifies (snq).
By Proposition~\ref{propPropiedlcmg}, one has
  $$m_{p}\le A^{2}M_{p}^{1/p}\le A^{2}m_{p},$$
and we may apply~(\ref{equation.weak.UyM}) to deduce
  $$\left(\frac{1}{B}\right) U(p)\leq m_p\leq A^2 B  U(p), \qquad p\in\N.$$
 Now we choose $k\in\N$ large enough such that $k^{1/\ro}/(2AB)^2\geq 1$. By the regular variation of $U$,
there exists $p_1\in\N$ such that
 $$\frac{m_{kp}}{m_{p}} \geq \frac{U(kp)}{A^2 B^2 U(p)} \geq \frac{k^{1/\ro}}{ 2A^2B^2}\geq 2>1, \qquad p>p_1.$$
 Then $\M$ verifies that
 $$\liminf_{p\ri\oo} \frac{m_{kp}}{m_p}>1,$$
and by Proposition~\ref{propPropiedlcmg}(iii) we conclude.
 \end{proof}

  In particular, by this lemma we deduce that $\lim_{p\ri\oo}m^V_p=\oo$ for every function $V \in \Bet(\ga, \ro(t) )$, where $(m^V_p)_{p\geq0}$
  is the quotient sequence associated to $(M^V_p)_{p\geq0}$. So,
 the associated function to the sequence $(M^V_p)_{p\geq0}$,  $M(t)=\sup_{p\in\N} \log (t^p/M^V_p)$,  and the counting
 function, $\nu(t)=\#\{j:m_j^V\le t\}$, are defined in $(0,\oo)$.

 To see that this construction is consistent, we want to prove that $V(t)$ and the associated function
 to the sequence $(M^V_p)_{p\geq0}$ are equivalent, i.e., $\lim_{t\ri\oo} V(t)/M(t)=1$. We will need some of the basic properties of the Young conjugate of a convex function  (see \cite{niculescu}, \cite{rockafellar}).
First we consider the following generalized definition of a convex function.

 \begin{den}\label{defConvex}
  Let $I\en\R$ be an interval and $f:I\ri\overline{\R}$, where $\overline{\R}:=\R\cup\{-\oo,\oo\}$ (with the usual extension of the order relation to this set).
  We say that $f$ is \textit{convex} if
  $$f((1-\lambda)x+\lambda y)\leq (1-\lambda)\a+\lambda \b $$
  whenever $f(x)<\a$ and $f(y)<\b$, for every $x,y\in I$ and every $\lambda\in(0,1)$. A function $g:I\ri\overline{\R}$ is said to be \textit{concave} if the function $-g$ is convex
  (we consider $-(\oo):=-\oo$ and $-(-\oo):=\oo$).
 \end{den}

These definitions agree with the classical ones when we consider $f,g:I\en\R\ri\R$.

 \begin{den}
  Let $f:\R\ri\overline{\R}$ be a convex function, the \textit{Young conjugate of $f$} is defined by
  $$f^*(y):=\sup_{x\in\R} (xy-f(x)), \qquad y\in\R.$$
 \end{den}

The following result, which may be deduced from more sophisticated ones in the book of R. T. Rockafellar~\cite[Ch.\ 7,\ 12]{rockafellar}, will be used in the next arguments.

 \begin{prop}\label{propdualconvexidadConjugadaYoung}
Let $f:\R\ri\R$ be a convex function. Then,  $f^*$ is convex and one has $(f^*)^*=f$.
\end{prop}

Let $\ro(t) \ri \ro>0$ be a proximate order,  $\ga>0$ and $V \in \Bet(\ga, \ro(t) )$. By property (IV) in Theorem~\ref{theorpropanalproxorde}
we know that the function $\fy_V\colon \R \ri\R$ defined by $\fy_V(x)=V(e^x)$ is strictly convex. So we can consider its Young conjugate
$$\fy^*_V(y)=\sup_{x\in\R} (xy-\fy_V(x)) = \sup_{x\in\R} (xy-V(e^x))= \sup_{t>0}(y \log(t)-V(t)).$$
We observe that
\begin{equation}\label{equaMpVExpYoungConj}
M^V_p=\sup_{t>0} \frac{t^p}{e^{V(t)}} = \exp \big(\sup_{t> 0}( p \log(t)-V(t))\big) =\exp{\fy^*_V(p)}, \qquad p\in\N_0.
\end{equation}
Moreover, since $\fy_V$ is continuous, Proposition~\ref{propdualconvexidadConjugadaYoung} implies that
\begin{equation}\label{equaDoubleYoungConj}
\fy^{**}_V:=(\fy^*_V)^*=\fy_V.
\end{equation}

Please note that, although the function $\fy_V$ only assumes real values, the function $\fy^*_V$ does take the value $+\infty$ in $(-\infty,0)$, and so we need to consider the extended Definition~\ref{defConvex} of a convex function.

 The following theorem relates $V(t)$ and $M(t)$ through the Young conjugate. The proof is based on a result from
 the PhD thesis of the third author (see \cite[Theorem\ 4.0.3]{schindlthesis}).

 \begin{theo}\label{theo.VequivM}
  Let $\ro(t) \ri \ro>0$ be a proximate order,  $\ga>0$ and $V \in \Bet(\ga, \ro(t) )$. If $M(t)$ is the associated function to the sequence $(M^V_p)_{p\geq0}$,
  then
  $$\lim_{t\ri\oo} \frac{V(t)}{M(t)}=1.$$
\end{theo}

\begin{proof}
  For simplicity, we note $(M_p)_{p\geq0}$ for the sequence $(M^V_p)_{p\geq0}$. We observe that, using~\eqref{equaMpVExpYoungConj} and~\eqref{equaDoubleYoungConj},
\begin{align}
M(t)&=\sup_{p\in\N_0} \log \left(\frac{t^p}{M_p}\right)= \sup_{p\in\N_0}  \left(p \log(t) - \fy^*_V(p)  \right) \no\\
  &\leq
  \sup_{y\in\R}\left(y \log(t) - \fy^*_V(y)  \right)= \fy^{**}_V(\log(t))= \fy_V(\log(t))=V(t).\label{desMyV1}
\end{align}
 Let us show that
\begin{equation*}
 \frac{\sup_{y\in\R} \left(y \log(t) - \fy^*_V(y)  \right) }{\sup_{p\in\N_0} \left(p \log(t) - \fy^*_V(p)  \right)}
 =\frac{V(t)}{M(t)} \leq \frac{\nu(t)}{\nu(t)-1}, \qquad  t>m_1
 \end{equation*}
(observe that for $t>m_1$, we have that $\nu(t)\geq2$ and $M(t)>0$). For $t>m_1$ we define the function
 $$f_t(y):= y \log(t)-\fy^*_V(y), \qquad y\in\R.$$
 Since $\fy^*_V(x)$ is convex (by Proposition~\ref{propdualconvexidadConjugadaYoung}),
 we have that $f_t(y)$ is a concave function.
 We see that $\fy^*_V(y)=\oo$ for $y<0$, because
 $$xy-V(e^x)<0\quad\textrm{for }x\geq 0,\qquad \textrm{and }\quad xy-V(e^x)\geq xy-V(1)\quad\textrm{for }x<0.
 $$
 Consequently, $f_t(y)=-\oo$ if $y<0$. As $\fy^*_V(0)=\log M_0=0$, we have $f_t(0)=0$ for every $t>m_1$. Then
 $\sup_{y\in\R}f_t(y)$, which equals $V(t)$, is attained for some $y>0$.

 Let us see that this happens in the interval $(\nu(t)-1,\nu(t)+1)$.
 Otherwise,
 there are two possibilities: there exists $x\geq\nu(t)+1$,  respectively $x\leq\nu(t)-1$, such that $f_t(x)>f_t(\nu(t))$. Since $f_t$ is concave, all the points $(y,f_t(y))$ such that $y\in[\nu(t),x]$, resp. $y\in[x,\nu(t)]$, have to lie above the line which
 connects $(\nu(t), f_t(\nu(t)))$ and $(x,f_t(x))$. In particular, $f_t(\nu(t)+1) > f_t(\nu(t))$,
 resp. $f_t(\nu(t)-1) > f_t(\nu(t))$, and this is impossible because, by virtue of~\eqref{equaRelationM_nu}, we have
 $$f_t(\nu(t))=M(t)=  \sup_{p\in\N_0} \left(p \log(t) - \fy^*_V(p)  \right)\geq f_t(p), \quad p\in\N_0.$$
 Therefore, there exists $y_t\in(\nu(t)-1,\nu(t)+1)$ such that
 $$f_t(y_t)=\sup_{y\in\R} f_t(y)=V(t).$$
 We distinguish three cases: $y_t=\nu(t)$, $y_t\in(\nu(t)-1,\nu(t))$ or $y_t\in(\nu(t),\nu(t)+1)$. First, if $y_t=\nu(t)$ we have that
 $$\frac{V(t)}{M(t)} =\frac{f_t(y_t)}{f_t(\nu(t))} =1 \leq \frac{\nu(t)}{\nu(t)-1}.$$
  Secondly, if we assume that $y_t\in (\nu(t),\nu(t)+1)$, we consider the connecting line between $(0,f_t(0))=(0,0)$ and $(\nu(t),f_t(\nu(t)))$, which is given
 by the equation $z=(f_t(\nu(t))/\nu(t))w$. Let us see that every point $(y,f_t(y))$ with $y\in(\nu(t),\nu(t)+1)$ has to lie below this line: Otherwise, we would have
 $f_t(y)>(f_t(\nu(t))/\nu(t))y$ for some $y\in(\nu(t),\nu(t)+1)$. This would imply that $(f_t(y)/y)\nu(t)>f_t(\nu(t))$, which is a contradiction
 to the concavity of $f_t$, since it means that the point $(\nu(t),f_t(\nu(t)))$ lies below the line with equation $z=(f_t(y)/y)w$, which joins $(0,0)$ and $(y,f_t(y))$.

 In particular, as $y_t\in (\nu(t),\nu(t)+1)$, we have seen that
 $$\frac{V(t)}{M(t)} =\frac{f_t(y_t)}{f_t(\nu(t))}\leq \frac{(f_t(\nu(t))/\nu(t)) y_t}{f_t(\nu(t))}
 \leq \frac{f_t(\nu(t))((\nu(t)+1)/\nu(t))}{f_t(\nu(t))}\leq \frac{\nu(t)+1}{\nu(t)}\leq \frac{\nu(t)}{\nu(t)-1}.$$
  Thirdly, if $y_t\in (\nu(t)-1,\nu(t))$, we consider now the connecting line between $(0,0)$ and $(\nu(t)-1, f_t(\nu(t)-1))$, which is given by the
 equation $z=(f_t(\nu(t)-1)/(\nu(t)-1))w$. Reasoning as before, the point $(y_t,f_t(y_t))$ lies below this line. Hence we have shown that
 $$f_t(y_t)\leq \frac{f_t(\nu(t)-1)}{\nu(t)-1} y_t \leq f_t(\nu(t)-1) \frac{\nu(t)}{\nu(t)-1} \leq f_t(\nu(t)) \frac{\nu(t)}{\nu(t)-1}.$$
Consequently,
 $$\frac{V(t)}{M(t)} =\frac{f_t(y_t)}{f_t(\nu(t))} \leq \frac{\nu(t)}{\nu(t)-1}.$$
 Using~(\ref{desMyV1}) and the previous information  we see that
 $$1\leq\frac{V(t)}{M(t)}\leq \frac{\nu(t)}{\nu(t)-1}, \quad t>m_1.$$
Since $\nu(t)\ri\oo$ as $t\to\oo$, we conclude that
 $$\lim_{t\ri\oo} \frac{V(t)}{M(t)}=1.$$
\end{proof}

It is clear from the last theorem that the sequence $(M^V_p)_{p\in\N_0}$ admits a nonzero proximate order in the sense
 of Definition~\ref{def.admit.p.o}, since $\ro_V$ is a proximate order and
 $$
 \lim_{t\to\oo}\log(t)\big(\ro_V(t)-d_{\M}(t)\big)=0.
 $$
 But the function
 $d_{\M}(t)$  is not necessarily a proximate order itself. However, the next result
   shows that we can construct a sequence $\L$ equivalent to $(M^V_p)_{p\in\N_0}$ such that $d_{\L}$ is a nonzero proximate order.

\begin{prop}\label{pro.suc.equiv.d.proxorder}
   Let $\ro(t) \ri \ro>0$ be a proximate order,  $\ga>0$ and $V \in \Bet(\ga, \ro(t) )$. There exists a (lc) sequence $\L$, with quotients tending to infinity, such that $\L\approx(M^V_p)_{p\in\N_0}$ and $d_{\L}(t)$ is a nonzero proximate order.
\end{prop}
\begin{proof}
We write $M_p$ instead of $M_p^V$, for short. By Theorem~\ref{theorem.construct.sequences.from.prox.order} and Proposition~\ref{propPropiedlcmg} there exists $A>0$ such that
  $$m_{p}\le A^{2}M_{p}^{1/p}\le A^{2}m_{p},$$
and applying~(\ref{equation.weak.UyM}) we have
 \begin{equation}\label{strong.equiv.MandU}
  \frac{1}{B} U(p)\leq m_p\leq A^2 B  U(p), \qquad p\in\N.
 \end{equation}
 We consider the sequence $\L=(L_p)_{p\in\N_0}$ given by the sequence of quotients $\bl=(\ell_p)_{p\in\N_0}$, with   $\ell_0:=U(1)$, $\ell_p:=U(p)$ for every $p\in\N$. Then
 $$L_0=1,\qquad L_p=U(1)\prod_{k=1}^{p-1} U(k), \quad p\in\N.$$
 Using~(\ref{strong.equiv.MandU}), we see that $\bl\simeq\bm$, and so $\L \approx \M$ by Proposition~\ref{propRelacionOrdenes}. Since the function $U(s)$ is increasing to infinity in $(0,\oo)$, the sequence of quotients $(\ell_p)_{p\in\N_0}$ also is, and $\L$ is (lc).
By the regular variation of the function $U(s)$, we see that
 the sequence $(\ell_p)_{p\in\N_0}$ is regularly varying with index $1/\ro$, since
 $$\lim_{p\ri\oo} \frac{\ell_{k p}}{\ell_p} = \lim_{p\ri\oo} \frac{U(kp)}{U(p)}=k^{1/\ro}, \qquad k \geq2.$$
By Theorem~\ref{theorem.charact.prox.order.nonzero}, $d_{\L}$ is a nonzero proximate order.
 \end{proof}

\subsection{Characterization of the admissibility condition}

As a consequence of the results in the previous subsection, we can show that the weaker conditions that are sufficient for the construction of flat functions in optimal sectors are indeed the same.

\begin{theo}\label{theo.charact.admit.p.o}
Let $\M$ be (lc) and $\lim_{p\ri\oo} m_p=\oo$, then the following conditions are equivalent:
\begin{enumerate}
 \item[(e)] There exists a (lc) sequence $\L$, with quotients tending to infinity, such that $\L\approx\M$ and $d_{\L}(t)$ is a nonzero proximate order.
 \item[(f)] $\M$ admits a nonzero proximate order.
\end{enumerate}
\end{theo}

\begin{proof}
 (e) $\Rightarrow$ (f)  If $M(t)$ and $L(t)$ are the associated functions to the sequences $\M$ and $\L$, as
 $\L\approx\M$, there exist positive constants $A$ and $B$  such that for every $t\in(0,\oo)$ one has
 $$L(At)\leq M(t)\leq L(Bt).$$
 Since $d_{\L}(t)$ is a nonzero proximate order, $L(t)=t^{d_{\L}(t)}$ is regularly varying by Proposition~\ref{propRegVarOrdAprox}, and we deduce that
 there exist positive constants $C$ and $D$ such that
 $$C \leq \frac{M(t)}{L(t)}\leq D \qquad \textrm{for $t$ large enough}.$$
 Finally, taking logarithms, we conclude that $\M$ admits a nonzero proximate order.

 (f) $\Rightarrow$ (e) Let $\ro(t)$ be the proximate order that $\M$ admits. There exist $A,B\in\R$ such that
 $$
 A\le \log(t)\big(d_{\M}(t)-\ro(t)\big)\le B\quad\textrm{for $t$ large enough}
 $$
or, in other words, there exist positive constants $A_0$ and $B_0$ such that
\begin{equation}\label{equaMboundedbytrot}
A_0\le \frac{M(t)}{t^{\ro(t)}}\le B_0\quad\textrm{for $t$ large enough}.
\end{equation}
Given $\ga>0$ we take a function $V \in \Bet(\ga, \ro(t))$. Since $\ro_V(t)=\log(V(t))/\log(t)$ and $\ro(t)$ are equivalent proximate orders, we know that
\begin{equation}\label{equaVequivtrot}
\lim_{t\to\oo}\frac{V(t)}{t^{\ro(t)}}=1,
\end{equation}
and from \eqref{equaMboundedbytrot} and \eqref{equaVequivtrot} we conclude that there exist positive constants $C$ and $D$ such that
 $$C\leq \frac{M(t)}{V(t)}\leq D, \qquad t\textrm{ large enough}.$$
 Using the regular variation of $V$, it is easy to show that there exist positive constants $E,F$
such that
\begin{equation}\label{equation.MyV.constantes}
 V(Et) \leq M(t)\leq V(Ft), \qquad t\textrm{ large enough}.
\end{equation}
Now, observe that, by the property (lc) of $\M$ and the very definition of the sequence $(M^V_p)_{p\in\N_0}$, we have that
 $$ M_p =\sup_{t>0} \frac{t^p}{e^{M(t)}}, \qquad M^V_p=\sup_{t>0} \frac{t^p}{e^{V(t)}},   \qquad p\in\N_0.$$
Then, (\ref{equation.MyV.constantes}) shows that there exists $p_0\in\N_0$ such that
 $$\frac{M^V_p}{F^p}  \leq M_p\leq \frac{M^V_p}{E^p}, \qquad p\geq p_0,$$
and we deduce that $\M\approx(M^V_p)_{p\in\N_0}$.
Finally, by Proposition~\ref{pro.suc.equiv.d.proxorder} we know that there exists a (lc) sequence $\L$, with quotients tending to infinity, such that $\L\approx(M^V_p)_{p\geq0}$ and $d_{\L}(t)$ is a nonzero proximate order.  Since we obviously have that $\L\approx\M$, we may conclude.
\end{proof}

\begin{rema}
The implication (a) $\Rightarrow$ (e) (see Theorems~\ref{theorem.charact.prox.order.nonzero} and~\ref{theo.charact.admit.p.o}) is obvious, while Example~\ref{example.dM.not.p.o} shows that the converse fails.\par
It is also clear that (e) $\Rightarrow$ (g), by Corollary~\ref{coroConseqRegularVariation} and since the existence, and the value if it exists, of the limit in~(\ref{limit.logmp.over.logp.omega}) are stable under equivalence, and the same is true for the value of the indices $\ga(\M)$ and $\o(\M)$.
Again, the converse implication (g) $\Rightarrow$ (e) fails, as the next Example~\ref{exampleGdoesnotimplyE} shows.\par
So, (a) $\Rightarrow$ (e) $\Rightarrow$ (g), and the arrows cannot be reversed. We think it is an interesting open problem to look for some easy-to-check condition (h) on the sequence $\M$ such that one has that (e) $\Leftrightarrow$ [(g)+(h)]. We note that the condition (h) should be stable by equivalence of sequences.

\end{rema}

\begin{exam}\label{exampleGdoesnotimplyE}
Let $\M$ be defined using the sequence of  quotients
$(m_p)_{p\in \N_0}$. The construction is similar as the one  given in Example~\ref{example.dM.not.p.o}.
We set $m_0=m_1=1$ and $m_2=2$. For each $k\in\N_0$, we consider the interval $(2^{2^k}, 2^{2^{k+1}}]$ which we divide in $2^k$
 subintervals. We put $I^k_j:=(2^{2^k+j}, 2^{2^k+j+1}]$ for $0\le j\le 2^k-1$.
For $0\le j\le k-1$  we
define $m_p$ as follows:
$$m_p:= 4 m_{2^{2^k+j}}, \qquad p\in I^k_j,$$
and for $k\leq j\leq 2^k-1$ we set
$$m_p:= 2^{\tau_k} m_{2^{2^k+j}}, \qquad p\in I^k_j, $$
where $\tau_k= (2^k-2k)/(2^k-k) $. For $k\in\N$, we observe that
$$m_{2^{2^k}}=2^{2^k}\qquad\text{and} \qquad  m_{2^{{2^k}+ k}}=2^{{2^k}+2k}.$$
The sequence $\M$ is clearly (lc), since $(m_p)_{p\in\N_0}$ is nondecreasing, and $\lim_{p\to\oo}m_p=\oo$. By definition, for any
$p\in I^k_j$ we have that $2p$ belongs to the adjacent interval $I^k_{j+1}$. We distinguish two cases:
\begin{enumerate}
 \item If $p\in I_j^k$ for $0\leq j\leq k-2$ or $j=2^k-1$, we have that $m_{2p}/m_p= 4$.
 \item If $p\in I_j^k$ for $k-1\leq j\leq 2^k-2$,  we have that
 $m_{2p}/m_p=2^{\tau_k}$.
\end{enumerate}
We observe that $\lim_{k\ri\oo} \tau_k=1$.
From both cases, we have that
\begin{equation}\label{equa.Schindl.Exam.liminf.limsup}
 1<2=\liminf_{p\ri\oo}\frac{m_{2p}}{m_p} \leq \limsup_{p\ri\oo}\frac{m_{2p}}{m_p}=4<\oo.
 \end{equation}
Using Proposition~\ref{propPropiedlcmg}, items (ii.2) and (iii), we see that $\M$ is (mg) and (snq). Next, we are
going to show that
$$\lim_{p\to\oo} \frac{\log(m_p)}{\log(p)}=\o(\M)=1.$$
If $p\in I^k_j$ for $0\le j\le k-1$  we get
$$\frac{(j+1)\log(4)+2^k\log(2)}{(2^k+j+1)\log(2)}\le\frac{\log(m_p)}{\log(p)}\le\frac{(j+1)\log(4)+2^k\log(2)}{(2^k+j)\log(2)}.$$
Since $0\le j\le k-1$ we have that
\begin{equation}\label{equa.Schindl.Exam.limlogmp.0.k}
\frac{2+2^k}{2^k+k}\le\frac{\log(m_p)}{\log(p)}\le\frac{2k +2^k }{2^k}.
\end{equation}
Now let $p\in I^k_{k+j}$ for $0\le j\le 2^k -k-1$, then
$$\frac{(j+1)\cdot\tau_k\cdot\log(2)+(2^k+2k)\log(2)}{(2^k+k+j+1)\log(2)}\le \frac{\log(m_p)}{\log(p)}\le \frac{(j+1)\cdot\tau_k\cdot\log(2)+(2^k+2k)\log(2)}{(2^k+k+j)\log(2)},$$
or equivalently,
$$1+\frac{k(2^k-k-j-1)}{(2^k+k+j+1) (2^k-k) }\le\frac{\log(m_p)}{\log(p)}\le1+\frac{k(2^k-k-j-2)+2^k}{(2^k+k+j) (2^k-k) }.$$
As $0\le j\le 2^k -k-1$, we see that
\begin{equation}\label{equa.Schindl.Exam.limlogmp.k.2^k}
1\le\frac{\log(m_p)}{\log(p)}\le1+\frac{k(2^k-k-2)+2^k}{2^{2k}-k^2 }.
\end{equation}
By~\eqref{equa.Schindl.Exam.limlogmp.0.k} and~\eqref{equa.Schindl.Exam.limlogmp.k.2^k} we see that
$\lim_{p\to\oo}\log(m_p)/\log(p)=1$. Using Remark~\ref{remaAlmostIncrease} one can show that $\gamma(\M)=1=\o(\M)$, then $\M$ satisfies (g).
As it happens in Example~\ref{example.dM.not.p.o},
by~\eqref{equa.Schindl.Exam.liminf.limsup}, the limit $\lim_{p\ri\oo}m_{2p}/m_p$
does not exist, then condition (d) in Theorem~\ref{theorem.charact.prox.order.nonzero} is violated.

Let us see that
condition (e) is also violated. On the contrary, suppose there exists
some other
(lc) sequence $\L$ with $\lim_{p\rightarrow\infty}\ell_p=+\infty$ such that $\M\approx\L$ and $d_{\L}$ is a nonzero proximate order.
By Proposition~\ref{propRelacionOrdenes}, as $\L$ is
(mg), we deduce that $\bm\simeq\bl$. Then there exists a constant $A\ge 1$ such that for all $q\in\N$, $q\ge 2$, and all
$p\in\N$ we get
$$\frac{1}{A^2}\cdot\frac{m_{qp}}{m_p}\le\frac{\ell_{qp}}{\ell_p},$$
and, by Theorem~\ref{theorem.charact.prox.order.nonzero}, the quotient on the right tends to $q^{\omega}$ for some $\omega=\omega(\L)>0$ as $p\rightarrow\infty$. In particular,
for every $i\in\N$ and $q=2^i$ we deduce that
\begin{equation}\label{equa.Schindl.Exam.limsup}
\frac{1}{A^2}\cdot\limsup_{p\rightarrow\infty}\frac{m_{2^ip}}{m_p}\le 2^{i\o}.
\end{equation}
Note that, since  $\bm\simeq\bl$, we have $\o=\o(\L)=\o(\M)=1$.
Moreover, for our sequence $\M$, given $i\in\N$ we have
for every $k\geq i$ that
$$\frac{m_{2^{{2^k}+ i}}}{m_{2^{2^k}}}=\frac{2^{2^k+2i}}{2^{2^k}}=4^i,$$
and this makes~\eqref{equa.Schindl.Exam.limsup} impossible.
\end{exam}

\begin{rema}
We summarize our previous examples. According to the previous remark, the sequences $\M_q$ and $\M_{0,\b}$ do not admit a nonzero proximate order, since they are not strongly regular. Regarding strongly regular sequences $\M$, for those appearing in applications $d_{\M}$ is a nonzero proximate order. This does not hold for the sequence in Example~\ref{example.dM.not.p.o}, which however does admit a nonzero proximate order. The previous example shows that there exist strongly  regular sequences satisfying (g) and not admitting a nonzero proximate order. Finally, the next example will provide a extremely pathological situation in which the limit in~(\ref{limit.logmp.over.logp.omega}) does not even exist.

In particular, strong regularity and~(\ref{limit.logmp.over.logp.omega}) are independent conditions, and both are necessary if we want $\M$ to admit a nonzero proximate order.

Also, note that for a strongly regular sequence not admitting a nonzero proximate order (like those in Examples~\ref{exampleGdoesnotimplyE} and~\ref{exampleNoAdmiteOrden}), the technique of construction of nontrivial flat functions in sectors of optimal opening $\pi\o(\M)$, developed in~\cite{SanzFlat}, is not available. It is worth mentioning that, although V.Thilliez~\cite{thilliez} is able to construct such flat functions for any strongly regular sequence, the sectors in which they are defined have their opening strictly smaller than $\pi\ga(\M)$. Since one always has $\ga(\M)\le \o(\M)$ (see~\cite[Prop.\ 3.7]{SanzFlat}), the opening is not optimal in his case.
\end{rema}

\begin{exam}\label{exampleNoAdmiteOrden}
It was an open question whether strong regularity was enough to have condition~(\ref{limit.logmp.over.logp.omega}). We answer this question in the negative.

We define $\M$ by the sequence of its quotients,
 $$m_0:=1, \qquad m_p:=e^{\de_p/p}m_{p-1}=\exp\left(\sum^p_{k=1}\frac{\de_k}{k} \right), \qquad p\in\N.$$
We consider the sequences
 $$k_{n}:=2^{3^n} <q_{n}:=k^2_n=2^{ 3^n 2}< k_{n+1}=2^{3^{n+1}}, \quad{n\in\N_0},$$
and we choose the sequence $(\de_k)^\oo_{k=1}$ as follows:
\begin{align*}
\de_1&=\de_2=2, \\
\de_k&=3,\quad \text{if} \quad k\in \{k_j+1, \dots, q_j\}, j\in\N_0,\\
\de_k&=2,\quad \text{if} \quad k\in \{q_j+1, \dots, k_{j+1}\}, j\in\N_0.
\end{align*}
From the very definition we deduce immediately that $m_{p+1}>m_p$ for $p\in\N_0$, then $\M$ is (lc).
We clearly have that
$$\exp\left(2 \sum^{2p}_{k=p+1}\frac{1}{k} \right)\leq\frac{m_{2p}}{m_p}=\exp\left( \sum^{2p}_{k=p+1}\frac{\de_k}{k} \right)
\leq \exp\left( 3\sum^{2p}_{k=p+1}\frac{1}{k} \right).$$
Using the asymptotic expression~\eqref{equaPartialSumsHarmonic} for the partial sums of the harmonic series,
we have that
$$\exp\left(2 \log(2)+2\ep_{2p}-2\ep_p\right)\leq\frac{m_{2p}}{m_p}
\leq \exp\left( 3 \log(2)+3\ep_{2p}-3\ep_p\right).$$
From these inequalities and using Proposition~\ref{propPropiedlcmg}, we deduce that $\M$ satisfies (mg) and (snq), therefore $\M$ is strongly regular. Observe that $\M$ verifies~(\ref{limit.logmp.over.logp.omega}) if, and only if, the sequence
$$t_p:=\frac{1}{\log(p)}\sum^p_{k=1}\frac{\de_k}{k},\qquad p\in\N,$$
is convergent (in other words, precisely when the sequence $(\de_k)^\oo_{k=1}$ is Riesz summable, see~\cite{Boos}).
We have the following relations:
\begin{align*}
t_{q_n}&=\frac{\log(k_{n})}{\log(q_n)}t_{k_n}+ 3 \frac{H_{q_n}-H_{k_n}}{\log(q_n)}=\frac{t_{k_n}}{2}+ \frac{3}{2}
+3\frac{\ep_{q_n}-\ep_{k_n}}{\log(q_n)},\\
t_{k_{n+1}}&=\frac{2t_{q_n}}{3}+ \frac{2}{3}+2\frac{\ep_{k_{n+1}}-\ep_{q_n}}{\log(k_{n+1})}.
\end{align*}
From these formulas, it is easy  to check that either $(t_{k_n})^\oo_{n=0}$
and $(t_{q_n})^\oo_{n=0}$ have different limits, or none of them converge.
In both situations, the sequence $(t_p)_{p\in\N}$ has no limit, and therefore $(\de_k)^\oo_{k=1}$ is not Riesz summable, which
leads to the conclusion.
\end{exam}

\begin{rema}
As soon as $\M$ admits a nonzero proximate order, the indices $\ga(\M)$ and $\o(\M)$ agree, so this happens for the sequences $\M_{\a,\b}$ (with $\a>0$) and for the one in Example~\ref{example.dM.not.p.o}. In turn, in the preceding example we have $\ga(\M)=2<\o(\M)=5/2$. The value of $\ga(\M)$ can be determined taking into account Remark~\ref{remaAlmostIncrease}, while the value of $\o(\M)$ requires a careful study of the Riesz means $(t_p)_{p\in\N}$. This is, to our knowledge, the first example of strongly regular sequence with different indices. In this situation, although non-quasianalyticity in $\widetilde{\mathcal{A}}_{\M}(G)$ is known to hold for every sectorial region $G$ with opening $\pi\gamma$ with $\gamma<\o(\M)$, no explicit construction of nontrivial flat functions in $\widetilde{\mathcal{A}}_{\M}(G)$ is currently available in case $\ga\in[\ga(\M),\o(\M))$.
\end{rema}

\noindent\textbf{Acknowledgements}: The first two authors are partially supported by the Spanish Ministry of Economy and Competitiveness under project MTM2012-31439. The first author is partially supported by the University of Valladolid through a Predoctoral Fellowship (2013 call) co-sponsored by the Banco de Santander. The third author is partially supported by the FWF-Project P 26735--N25 and by the FWF-Project J 3948--N35.

\vskip.5cm
\noindent Authors' Affiliation:\par\vskip.5cm
Javier Jim\'enez-Garrido and Javier Sanz\par
Departamento de \'Algebra, An\'alisis Matem\'atico, Geometr{\'\i}a y Topolog{\'\i}a\par
Instituto de Investigaci\'on en Matem\'aticas de la Universidad de Valladolid, IMUVA\par
Facultad de Ciencias\par
Universidad de Valladolid\par
47011 Valladolid, Spain\par
E-mail: jjjimenez@am.uva.es (Javier Jim\'enez-Garrido), jsanzg@am.uva.es (Javier Sanz).
\par\vskip.5cm

Gerhard Schindl\par
Fakult\"at f\"ur Mathematik\par
Universit\"at Wien\par
Oskar-Morgenstern Platz 1, A-1090 Wien, Austria\par
E-mail: gerhard.schindl@univie.ac.at.

\end{document}